\documentclass{article}

\usepackage{amsthm,amsmath,natbib}

\usepackage{amssymb}
\usepackage{graphicx}
\usepackage{ulem}
\usepackage{bm}
\usepackage{color}

\usepackage{natbib}
\usepackage{url}
\usepackage{enumerate}
\usepackage{rotating}

\usepackage{dsfont}

\usepackage{accents}

\theoremstyle{plain}

\newtheorem{lemma}{Lemma}[section]
\newtheorem{theorem}[lemma]{Theorem}
\newtheorem{cor}[lemma]{Corollary}
\newtheorem{prop}[lemma]{Proposition}

\theoremstyle{definition}

\newtheorem{exam}{Example}[section]

\newcommand{\R}{\mathbb{R}}

\newcommand{\norm}[1]{\left\Vert#1\right\Vert}

\newcommand{\abs}[1]{\left\vert#1\right\vert}
\newcommand{\set}[1]{\left\{#1\right\}}

\newcommand{\bfx}{\bm{x}}
\newcommand{\bfzero}{\bm{0}}

\newcommand{\bfone}{\bm{1}}
\newcommand{\bfa}{\bm{a}}
\newcommand{\bfb}{\bm{b}}
\newcommand{\bfc}{\bm{c}}

\newcommand{\bfs}{\bm{s}}

\newcommand{\bft}{\bm{t}}

\newcommand{\bfu}{\bm{u}}

\newcommand{\bfv}{\bm{v}}

\newcommand{\bfX}{\bm{X}}
\newcommand{\bfY}{\bm{Y}}
\newcommand{\bfy}{\bm{y}}
\newcommand{\bfZ}{\bm{Z}}
\newcommand{\bfz}{\bm{z}}

\newcommand{\bfmu}{\bm{\mu}}

\newcommand{\bfxi}{\bm{\xi}}

\newcommand{\indic}{ {{\rm 1}\hskip-2.5pt{\rm l}}}

\begin{document}

\begin{center}
{\Large {\bf On a class of norms generated by nonnegative integrable distributions}}
\end{center}

\begin{center}
Michael Falk$^{a}$ and Gilles Stupfler$^{b}$ \\ 
\vspace*{0.2 cm}

\noindent
$^{a}$ Institute of Mathematics, University of W\"{u}rzburg, W\"{u}rzburg, Germany \\ 
$^{b}$ School of Mathematical Sciences, University of Nottingham, United Kingdom \\
\end{center}

{\bf MSC2010 subject classification:} 60E10, 60G99, 62H05, 62H12. \\


\begin{abstract} 
\noindent
We show that any distribution function on $\R^d$ with nonnegative, nonzero and integrable marginal distributions can be characterized by a norm on $\R^{d+1}$, called $F$-norm. We characterize the set of $F$-norms and prove that pointwise convergence of a sequence of $F$-norms to an $F$-norm is equivalent to convergence of the pertaining distribution functions in the Wasserstein metric. On the statistical side, an $F$-norm can easily be estimated by an empirical $F$-norm, whose consistency and weak convergence we establish. \vskip1ex
\noindent
The concept of $F$-norms can be extended to arbitrary random vectors under suitable integrability conditions fulfilled by, for instance, normal distributions. The set of $F$-norms is endowed with a semigroup operation which, in this context, corresponds to ordinary convolution of the underlying distributions. Limiting results such as the central limit theorem can then be formulated in terms of pointwise convergence of products of $F$-norms. 
\vskip1ex
\noindent
We conclude by showing how, using the geometry of $F$-norms, we may characterize nonnegative integrable distributions in $\R^d$ by simple compact sets in $\R^{d+1}$. We then relate convergence of those distributions in the Wasserstein metric to convergence of these characteristic sets with respect to Hausdorff distances.
\end{abstract}

\section{Introduction}

It was observed only recently that a particular kind of norms on $\R^d$, called \textit{$D$-norms}, are the skeleton of multivariate extreme value theory. Deep results like Takahashi's characterizations~\citep[see][]{takahashi1987,takahashi1988} of multivariate max-stable distributions with independent or completely dependent margins by the value of their extremal coefficient, and classifications of multivariate dfs in terms of their multivariate extreme value domains of attraction~\citep{deh84,gal87} turn out to be easily seen properties of $D$-norms. The framework of $D$-norms has also recently been used to design simulation techniques for max-stable processes~\citep{falhofzot2015,falzot2017} and prove new results on multivariate records~\citep{domfalzot2018,domzot2018}. The concept of $D$-norms can be extended to define norms on functional spaces, and mathematically complex results such as the classification of simple max-stable distributions in spaces of continuous functions~\citep{ginhahvat1990} can be rewritten elegantly in the framework of functional $D$-norms. In addition, $D$-norms simultaneously provide a mathematical topic, which can be studied independently: an early, short introduction is~\citet{fahure10}, and an up-to-date account of $D$-norms is~\citet{falk2019}.

$D$-norms are defined via a random vector (rv), called {\it generator}. The distribution function (df) of this rv, however, is not uniquely determined, and there exists an infinite number of generators of the same $D$-norm. It was shown by \citet{falkst16} that the $D$-norm characterizes the distribution of a generator if the constant function one is added to the generator as a further component. This led to the definition of the \textit{max-characteristic function}, which can be used to identify the distribution of any multivariate distribution with nonnegative and integrable components. This notion of max-characteristic function is particularly interesting when considering standard extreme value distributions such as the Generalized Pareto distribution, for which it has a simple closed form, although the standard characteristic function based on taking a Fourier transform does not. However, the max-characteristic function does not define a norm and therefore loses, compared to $D$-norms, a number of interesting algebraic and geometric properties.

In this paper we build on these observations and construct a norm on $\R^{d+1}$, called $F$-norm, which contains the notion of max-characteristic function. In Section~\ref{defi}, we present the concept of $F$-norms, and show that the df of each rv $\bfX=(X_1,\dots,X_d)$ on $\R^d$ with nonnegative, nonzero and integrable components can be characterized by the pertaining $F$-norm. We then list examples and derive basic properties as well as an inversion formula to retrieve a distribution from its associated $F$-norm. We also fully characterize the set of $F$-norms and obtain a simple classification in two dimensions. In Section~\ref{FnormDnorm} we investigate the similarities and the differences between $F$-norms and $D$-norms in detail. We show in particular that the extremal coefficient of a multivariate extreme value copula, which can be written in terms of a $D$-norm, can be recovered in a simple fashion from the $F$-norm the copula generates. This suggests that statistically important quantities such as the extremal coefficient can be inferred by estimating $F$-norms.

In Section \ref{seclim} we carry this idea forward and analyse the convergence of sequences of $F$-norms. We start by proving that pointwise convergence of a sequence of $F$-norms to an $F$-norm is equivalent with convergence of the pertaining dfs with respect to the Wasserstein metric. We then add some statistical views on $F$-norms to this section. The (random) $F$-norm $\norm\cdot_{\widehat F_n}$ of the empirical df $\widehat F_n$ of a sample of $n$ independent and identically distributed (iid) rvs is an estimator of $\norm\cdot_F$ with the structure of a sample mean. Local uniform consistency and asymptotic normality of $\norm\cdot_{\widehat F_n}$ as an estimator of $\norm\cdot_F$ are then consequences of the law of large numbers and the multivariate central limit theorem. More strongly, we establish the $\sqrt{n}$-functional weak convergence of $\norm\cdot_{\widehat F_n} - \norm\cdot_F$ to a Gaussian process which is essentially a functional of a Brownian bridge.

Section~\ref{seclim} suggests that $F$-norms interact nicely with well-known modes of convergence and theorems of statistical analysis. In order to be able to use these norms in practice for asymptotic analyses, it is important to understand how they behave with respect to simple algebraic operations. It turns out that two $F$-norms can be multiplied by constructing the $F$-norm generated by the componentwise product of pairs of {\it independent} rvs giving rise to the individual $F$-norms. We also provide an integral formula making it possible, given two $F$-norms, to compute this product in a straightforward way. Equipped with this commutative multiplication, the set of $F$-norms is a semigroup with an identity element, and we can fully identify the invertible and idempotent elements for this operation. This algebraic aspect is investigated in Section \ref{sec:algebra}.

The concept of $F$-norms as we introduce it originally focuses on multivariate rvs with nonnegative and integrable components, and thus excludes common distributions such as the multivariate normal distribution. In Section \ref{sec:generalrv} we show that we can also define, by an exponential transformation, a concept of $F$-norms for a rv attaining negative values, under an integrability condition. This indeed allows us to include multivariate normal distributions, as well as other interesting examples. The multiplication of $F$-norms in Section \ref{sec:algebra} then represents the convolution of two rvs, and central limit theorems for iid rvs now mean pointwise convergence of the sequence of corresponding products of $F$-norms.

A multivariate distribution can then, under an integrability assumption, be characterized by its associated $F$-norm. The norm structure makes it possible to reduce the knowledge of the df $F$ to even simpler objects than the full $F$-norm. Because each norm is a homogeneous function, the knowledge of an $F$-norm (and thus of the underlying df $F$) is equivalent to its knowledge on the unit simplex. Besides, and since a norm is characterized by its unit sphere, multivariate distributions on $\R^d$ can be characterized, under suitable integrability conditions on the components, by the part of the unit sphere for their $F$-norm contained in the positive orthant of $\R^{d+1}$, which is a compact set. Interestingly, the convergence of $F$-norms, and therefore convergence of $d$-dimensional distributions in the Wasserstein metric, can be shown to be equivalent to the convergence of these unit spheres with respect to any Hausdorff metric induced by a norm in $\R^{d+1}$. These geometric aspects are investigated in Section \ref{sec:geometry}.

\section{The concept of $F$-norms}
\label{concept}
\subsection{Definition, examples, and basic properties}
\label{defi}
Let $d\geq 1$ and $\bfX=(X_1,\dots,X_d)$ be a rv satisfying the fundamental assumption \\

\noindent
$(\mathcal{H})$ Each $X_i$ is almost surely (a.s.) nonnegative with $0<E(X_i)<\infty$. \\

\noindent
For $\bfx=(x_0,x_1,\dots,x_d)\in\R^{d+1}$, define a mapping $\norm{\cdot}_{\bfX}$ by
\begin{equation}\label{defn:definition of F-norm}
\norm{\bfx}_{\bfX}:= E\left(\max\left(\abs{x_0}, \abs{x_1}X_1,\dots,\abs{x_d}X_d\right)\right).
\end{equation}
This paper is based on the following fundamental observations, presented in the two subsequent results.
\begin{lemma}
\label{lemXnorm}
If $\bfX$ satisfies $(\mathcal{H})$ then $\norm{\cdot}_{\bfX}$ is a norm on $\R^{d+1}$.
\end{lemma}
\begin{proof} Clearly $\norm{\cdot}_{\bfX}$ is a well-defined and finite nonnegative function. Positive definiteness follows by noting that $\norm{\bfx}_{\bfX} = 0$ implies $x_0 = 0$ as well as $\max(\abs{x_1}X_1,\dots,$ $\abs{x_d}X_d) = 0$ almost surely, which in turn implies $x_1=\cdots=x_d=0$ because each $X_i$ is positive with nonzero probability. Homogeneity of $\norm{\cdot}_{\bfX}$ is obvious, and the triangle inequality simply follows from the usual triangle inequality for $|\cdot|$.
\end{proof}
It turns out that the norm $\norm{\cdot}_{\bfX}$ characterizes the df of $\bfX$. This is the content of our first main result, in which by the equality of two norms we mean their pointwise equality.
\begin{theorem}\label{theo:F-norm characterizes distribution}
Let $\bfX$ and $\bfY$ be rvs on $\R^d$, satisfying condition~$(\mathcal{H})$, with dfs $F$ and $G$. Then $F=G$ if and only if 
$
\norm\cdot_{\bfX}=\norm\cdot_{\bfY}.
$
\end{theorem}

\begin{proof}
The function $\varphi_{\bfX}$ defined for any $\bfx=(x_1,\ldots,x_d)\ge\bfzero\in\R^d$ by
\begin{equation}\label{defn:definition of max-CF}
\varphi_{\bfX}(\bfx):= E\left(\max(1,x_1X_1,\dots,x_dX_d) \right)
\end{equation}
is the {\it max-characteristic function} (max-CF) pertaining to $\bfX$ (any operation on vectors such as $+,\geq,\ldots$ is meant componentwise throughout). As shown by \citet[Lemma 1.1]{falkst16} it characterizes the distribution of $\bfX$. Since clearly $\norm\cdot_{\bfX} = \norm\cdot_{\bfY} \Rightarrow \varphi_{\bfX} = \varphi_{\bfY}$, this implies the assertion.
\end{proof}

In view of the above result we denote the norm $\norm\cdot_{\bfX}$ by $\norm\cdot_F$ when $\bfX$ has df $F$, and we call every norm on $\R^{d+1}$ which has the representation \eqref{defn:definition of F-norm} an {\it $F$-norm}. 

Let us point out that Theorem~\ref{theo:F-norm characterizes distribution} is still valid when $\bfX$ is not assumed to have nonzero components, but the mapping $\norm{\cdot}_{\bfX}$ is then actually only a {\it seminorm} on $\R^{d+1}$. Extending the definition of the max-CF of $\bfX$ by considering the mapping $\norm\cdot_F$ thus generally leads to a seminorm rather than a norm. Observe though that unless $\bfX$ is the degenerate rv $\bfzero\in \R^d$, the mapping $\norm{\cdot}_F$ induces an $F$-norm on $\R^{d'+1}$, where $d'$ is the number of nonzero components of $\bfX$. There is therefore no loss of generality in considering $F$-norms rather than $F$-seminorms, and we do so in the remainder of this paper.

An $F$-norm is usually conveniently calculated by using the following fundamental formula.
\begin{lemma}
\label{lemfund}
Let $F$ be the df of a rv $\bfX$ satisfying condition~$(\mathcal{H})$. Then, for any $\bfx=(x_0,x_1,\dots,x_d)\in\R^{d+1}$, we have
$$
\norm{\bfx}_F = |x_0| + \int_{|x_0|}^{\infty} [1-F(t/|x_1|,\ldots,t/|x_d|)] \, dt
$$
with the convention $1/0=\infty$.
\end{lemma}
\begin{proof} This is a straightforward consequence of the well-known formula
$$
E(|Z|) = \int_0^{\infty} P(|Z|>t) \, dt
$$
applied to the nonnegative rv $Z=\max\left(\abs{x_0}, \abs{x_1}X_1,\dots,\abs{x_d}X_d\right)$.
\end{proof}

\begin{exam}[Degenerate $F$-norm]
\label{exdeg}
\upshape The degenerate distribution concentrated at a $d$-dimensional vector $\bfc=(c_1,\ldots,c_d)>\bfzero$ is characterized by the $F$-norm
$$
\norm{\bfx}_F = \max(|x_0|,c_1|x_1|,\ldots,c_d|x_d|).
$$
In particular, the standard sup-norm $\norm{\bfx}_\infty := \max_{0\le i\le d}\abs{x_i}$ on $\R^{d+1}$ is an $F$-norm which characterizes the constant rv $(1,\dots,1)\in\R^d$.
\end{exam}

\begin{exam}[Bernoulli $F$-norm]
\label{exber}
\upshape The Bernoulli distribution with parameter $p\in (0,1)$ is characterized by the bivariate $F$-norm
$$
\norm{(x_0,x_1)}_F = (1-p) |x_0| + p\max(|x_0|,|x_1|).
$$
\end{exam}

\begin{exam}[Uniform $F$-norm]
\label{exunif}
\upshape The uniform distribution on $(0,1)$ is characterized by the bivariate $F$-norm
\[
\norm{(x_0,x_1)}_F=\begin{cases}
|x_0|,&\text{if } |x_1| \le |x_0|,\\
|x_0| + \displaystyle\int_{|x_0|}^{|x_1|} \left( 1-\dfrac{t}{|x_1|} \right) dt = \dfrac{x_0^2+x_1^2}{2|x_1|}, &\text{if }|x_1|>|x_0|.
\end{cases}
\]
\end{exam}

\begin{exam}[Exponential $F$-norm]
\label{exexpo}
\upshape The exponential distribution with mean $1/\lambda$, $\lambda>0$ is characterized by the bivariate $F$-norm
$$
\norm{(x_0,x_1)}_F = |x_0| + \int_{|x_0|}^\infty \exp\left(-\lambda \frac{t}{|x_1|} \right)\,dt= |x_0| + \frac{|x_1|}{\lambda}\exp\left(-\lambda\frac{|x_0|}{|x_1|} \right)
$$
when $x_1\neq 0$, and $|x_0|$ otherwise.
\end{exam}

\begin{exam}[Pareto $F$-norm]
\label{expareto}
\upshape The Pareto distribution with tail index $\gamma\in (0,1)$, having df $F(t)=1-t^{-1/\gamma}$, $t\geq 1$, is characterized by the bivariate $F$-norm
\begin{eqnarray*}
\norm{(x_0,x_1)}_F &=& |x_0| + \int_{|x_0|}^\infty \left[ \indic_{\{ t\leq |x_1|\}} + \left( \frac{t}{|x_1|} \right)^{-1/\gamma} \indic_{\{ t>|x_1|\}} \right] \,dt \\[5pt]
				   &=& \begin{cases} |x_0| & \mbox{if } |x_1| = 0, \\[5pt] \displaystyle |x_0| \left[ 1+\frac{\gamma}{1-\gamma} \left( \frac{|x_0|}{|x_1|} \right)^{-1/\gamma} \right] & \mbox{if } 0< |x_1|\leq |x_0|, \\[5pt] \displaystyle |x_1| \frac{1}{1-\gamma} & \mbox{if } |x_1|> |x_0|. \end{cases}
\end{eqnarray*}
\end{exam}

We now explore some simple properties of $F$-norms. Each $F$-norm induces, as a norm, a continuous function on $\R^{d+1}$. It takes the value $1$ at $(1,0,\ldots,0)$. It also defines a radially symmetric function, {\it i.e.}
\[
\forall \bfx\in\R^{d+1}, \ \norm{\bfx}_F = \norm{\abs{\bfx}}_F, \ \mbox{ with } \ \abs{\bfx} = (|x_0|,|x_1|,\ldots,|x_d|).
\]
The norm $\norm\cdot_F$ is, therefore, determined by its values on $[0,\infty)^{d+1}$. Additionally, any $F$-norm defines a monotone norm on $\R^{d+1}$ in the sense that
\[
\bfzero\le\bfx\le\bfy\Rightarrow \norm{\bfx}_F\le\norm{\bfy}_F.
\]
These properties make it possible, in some cases, to show that certain norms are not $F$-norms:
\begin{itemize}
\item the norm $\norm\cdot := 2\norm\cdot_\infty$ is not an $F$-norm because $\norm{(1,0,\ldots,0)} = 2$,
\item for any $\delta\in (0,1)$, the matrix
$$
M=\left( \begin{array}{cc} 1 & -\delta \\ -\delta & 1 \end{array} \right)
$$
is symmetric and positive definite, and therefore induces the norm
$$
\norm{(x_0,x_1)}_{\delta} := \left[ (x_0,x_1) M (x_0,x_1)^{\top} \right]^{1/2} = \sqrt{x_1^2- 2 \delta x_1 x_2+x_2^2}.
$$
This norm is not radially symmetric, as
$$
\norm{(1,-1)}_{\delta} = \sqrt{2}\sqrt{1+\delta} \neq \sqrt{2}\sqrt{1-\delta} = \norm{(1,1)}_{\delta}.
$$
It is actually not monotone either, since
$$
(1,0) \leq (1,\delta) \ \mbox{ but } \ \norm{(1,0)}_{\delta} = 1 > \sqrt{1-\delta^2} = \norm{(1,\delta)}_{\delta}.
$$
The norm $\norm{\cdot}_{\delta}$ therefore cannot be an $F$-norm.
\end{itemize}
We close this section by providing results to identify those norms which are $F$-norms. Let us highlight first that for any norm $\norm\cdot$ on $\R^{d+1}$ and any $\bfx\in \R^d$, the function $t\mapsto \norm{(t,\bfx)}$ is convex on $[0,\infty)$ (and right-continuous at 0), and therefore automatically absolutely continuous on this interval~\citep[see {\it e.g.}][]{rock70}. With this in mind, we have the following result.
\begin{theorem}
\label{theoclassif}
A norm $\norm\cdot$ on $\mathbb{R}^{d+1}$ is an $F$-norm if and only if the following two conditions hold:
\begin{itemize}
\item[(i)] it is radially symmetric,
\item[(ii)] there exists a rv $\bfX = (X_1,\ldots,X_d)$ which satisfies $(\mathcal{H})$ such that for any $x_1,\ldots,x_d>0$, the Lebesgue derivative of $t\mapsto \norm{(t,1/x_1,\ldots,1/x_d)}$ is equal to $P(X_1\leq tx_1,\ldots,X_d\leq tx_d)$ almost everywhere, and
$$
\norm{\left( 0,\frac{1}{x_1},\ldots,\frac{1}{x_d} \right)} = E\left( \max\left( \frac{X_1}{x_1},\ldots,\frac{X_d}{x_d} \right) \right).
$$
\end{itemize}
In that case then $\norm\cdot =\norm\cdot_F$ with $F$ being the df of $\bfX$.
\end{theorem}
\begin{proof} That any $F$-norm satisfies~(i) is obvious, while~(ii) is a clear consequence of Lemma~\ref{lemfund}, reformulated as
$$
\norm{(t,1/x_1,\ldots,1/x_d)}_F = t + \int_t^{\infty} \left[ 1- P\left( X_1\leq u x_1,\ldots,X_d\leq u x_d \right) \right] du
$$
when $\bfX$ has df $F$.

Conversely, let $\norm\cdot$ satisfy~(i) and~(ii). Since $\norm\cdot$ and $\norm{\cdot}_F$ are continuous, as well as radially symmetric by~(i), we only need to show that
$
\norm{\bfx} = \norm{\bfx}_F
$
for all $\bfx>\bfzero$. Pick such an $\bfx$ and write it as $\bfx=(t,1/x_1,\ldots,1/x_d)$, for $t,x_1,\ldots,x_d>0$. Write then, by absolute continuity,
\begin{eqnarray*}
\! \! \! \! & & \! \! \! \! \norm{(t,1/x_1,\ldots,1/x_d)} - \norm{(0,1/x_1,\ldots,1/x_d)} \\
\! \! \! \! &=& \! \! \! \! t - \! \int_0^t [1-P(X_1\leq u x_1,\ldots,X_d\leq u x_d)] \, du \\
\! \! \! \! &=& \! \! \! \! t + \! \int_{t}^{\infty} [1-P(X_1\leq u x_1,\ldots,X_d\leq u x_d)] \, du - E\left( \! \max\left( \frac{X_1}{x_1},\ldots,\frac{X_d}{x_d} \right) \! \right).
\end{eqnarray*}
Applying Lemma~\ref{lemfund} and noting that by (ii),
$$
\norm{\left( 0,\frac{1}{x_1},\ldots,\frac{1}{x_d} \right)} = E\left( \max\left( \frac{X_1}{x_1},\ldots,\frac{X_d}{x_d} \right) \right),
$$
concludes the proof.
\end{proof}
Although this result is hard to apply in arbitrary dimensions due to the high-level condition~(ii), it admits the following simple corollary in two dimensions.
\begin{cor}
\label{corclassif}
A norm $\norm\cdot$ on $\mathbb{R}^2$ is an $F$-norm if and only if the following two conditions hold:
\begin{itemize}
\item[(i)] it is radially symmetric,
\item[(ii)] the Lebesgue derivative of $t\mapsto \norm{(t,1)}$ is almost everywhere equal to a univariate df $F$ on $[0,\infty)$ with a finite first moment equal to $\norm{(0,1)}$.
\end{itemize}
In that case then $\norm\cdot =\norm\cdot_F$.
\end{cor}
\begin{exam}[On the $L^1$-norm]
\upshape The $L^1$-norm $\norm{(x_0,x_1)} = |x_0| + |x_1|$ on $\R^2$ is not an $F$-norm. Indeed, we have
$$
\dfrac{d}{dt} (\norm{(t,1)}) = 1, \ t>0,
$$
which does not define a df on $[0,\infty)$ having a (strictly) positive first moment.
\end{exam}
\begin{exam}[On the $L^p$-norm]
\upshape
Each $L^p$-norm $\norm{(x_0,x_1)}_p = (|x_0|^p+|x_1|^p)^{1/p}$ on $\R^2$, with $1<p<\infty$, is an $F$-norm. Indeed, it is clearly radially symmetric and
$$
\frac{d}{dt}\left( \norm{(t,1)}_p \right) = \frac{d}{dt}\left( (t^p + 1)^{1/p} \right) = (1+t^{-p})^{1/p - 1}, \ t>0,
$$
which defines the df of a Burr type III distribution in the sense of~\citet[Table 2.1]{beigoesegteu2004}. This distribution, for $p>1$, has a finite first moment.
\end{exam}
Even though providing a simple characterization of $F$-norms in arbitrary dimensions appears to be a difficult problem, there is a simple inversion formula inspired by Theorem~\ref{theoclassif} that makes it possible to go from an $F$-norm to its pertaining df. This is the focus of the following result, which can also be used to check that a norm is {\it not} an $F$-norm. Its proof is a straightforward consequence of Lemma~\ref{lemfund} and right-continuity of the df $F$.
\begin{cor}
\label{corinversion}
Let $\norm\cdot_F$ be an $F$-norm. Then, for any $x_1,\ldots,x_d>0$, the right-derivative of the function $t\mapsto \norm{(t,1/x_1,\ldots,1/x_d)}_F$ at $t=1$ exists and is $F(x_1,\ldots,x_d)$.
\end{cor}
\begin{exam}[On the $L^1$-norm again]
\label{exL1}
\upshape The $L^1$-norm 
$$
\norm{(x_0,x_1,\ldots,x_d)} = \sum_{i=0}^d |x_i|
$$ 
on $\R^{d+1}$ is not an $F$-norm. Indeed, we have, for any $x_1,\ldots,x_d>0$,
$$
\dfrac{d}{dt} (\norm{(t,1/x_1,\ldots,1/x_d)}) = 1, \ t>0,
$$
which defines the df of the degenerate vector $(0,\ldots,0)$. This distribution does not have strictly positive marginal moments and thus, by Corollary~\ref{corinversion}, $\norm\cdot$ cannot be an $F$-norm.
\end{exam}
\subsection{$F$-norms and $D$-norms}
\label{FnormDnorm}
$F$-norms are related to {\it $D$-norms}~\citep{falk2019}, which are defined as follows. Let $\bfZ=(Z_1,\dots,Z_d)$ be a componentwise nonnegative rv such that $E(Z_i)=1$, $1\le i\le d$. Then, by the arguments of Lemma~\ref{lemXnorm}, the quantity
$
\norm{\bfx}_D:= E\left(\max(\abs{x_1} Z_1,\ldots,\abs{x_d} Z_d) \right)
$
defines a norm on $\R^d$, called {\it $D$-norm}. The concept of $D$-norms has come to prominence recently for its importance in multivariate extreme value theory, not least because it allows for a simple characterization of max-stable dfs~\citep[Theorem 2.3.3]{falk2019}. The following example, which constructs the $F$-norm of a max-stable distribution, illustrates this further.
\begin{exam}[Max-stable $F$-norm]
\upshape
Let $G$ be a max-stable df on $\R^d$ with identical Fr\'{e}chet margins $G_i(x)=\exp\left(-x^{-p}\right)$, $x>0$, $p>1$. By \citet[Theorem 2.3.4]{falk2019} there exists a $D$-norm $\norm\cdot_D$ on $\R^d$ such that
\[
\forall \bfx=(x_1,\dots,x_d)>\bfzero, \ G(\bfx)=\exp\left(-\norm{\frac{\bfone}{\bfx^p}}_D \right).
\]
Recall that all operations on vectors are meant componentwise. Let the rv $\bfxi=(\xi_1,\dots,\xi_d)$ follow this df $G$. Apply Lemma~\ref{lemfund} to find that the $F$-norm on $\R^{d+1}$ induced by $G$ satisfies, for $(x_0,\dots,x_d)>\bfzero\in\R^{d+1}$,
\begin{align*}
\norm{(x_0,\dots,x_d)}_G &= x_0 + \int_{x_0}^\infty \left[ 1-\exp\left(-\frac{\norm{\bfx^p}_D}{t^p}\right) \right] \,dt\\
&= x_0 + \norm{\bfx^p}_D^{1/p}\int_{x_0/\norm{\bfx^p}_D^{1/p}}^{\infty} \left[ 1-\exp\left(-\frac 1{t^p} \right) \right] \,dt\\
&= \norm{\left(x_0,\norm{\bfx^p}_D^{1/p}\right)}_{F_p}
\end{align*}
where $\norm\cdot_{F_p}$ is the bivariate $F$-norm associated to the univariate Fr\'{e}chet df $F_p(x)=\exp\left(-x^{-p}\right)$, $x>0$, $p>1$. In particular, if $\xi_1,\dots,\xi_d$ are independent, we obtain $\norm\cdot_D=\norm\cdot_1$ and, thus,
\[
\norm{(x_0,\dots,x_d)}_G =  \norm{\left(x_0,\norm{\bfx}_p\right)}_{F_p}.
\]
\end{exam}
A consequence of Theorem~\ref{theo:F-norm characterizes distribution} is that the distribution of the {\it generator} $\bfZ$ of a $D$-norm whose first component is equal to 1 is characterized by this $D$-norm; this was already observed by \citet[Lemma 1.1]{falkst16} and led therein to the introduction of the max-CF as defined in~\eqref{defn:definition of max-CF}. Unlike for $F$-norms, however, the distribution of the generator of a $D$-norm is in general not uniquely determined. This is most easily seen through the following characterization of the set of generators of the sup-norm.
\begin{prop}
\label{propsupnorm}
The sup-norm $\norm{\cdot}_{\infty}$ is a $D$-norm, and $\bfZ$ generates $\norm{\cdot}_{\infty}$ as a $D$-norm if and only if $\bfZ = X(1,\ldots,1)$ a.s., where $X$ is a nonnegative rv having expectation 1.
\end{prop}
\begin{proof} That any such rv generates $\norm{\cdot}_{\infty}$ follows from the identity
\begin{eqnarray*}
E\left(\max(\abs{x_1} X,\ldots,\abs{x_d} X) \right) &=& E\left(X \max(\abs{x_1},\ldots,\abs{x_d}) \right) \\
													&=& \max(\abs{x_1},\ldots,\abs{x_d}).
\end{eqnarray*}
Conversely, suppose that $\bfZ=(Z_1,\dots,Z_d)$ is componentwise nonnegative, satisfies $E(Z_i)=1$, $1\le i\le d$, and
$$
E\left(\max(\abs{x_1} Z_1,\ldots,\abs{x_d} Z_d) \right) = \max(\abs{x_1},\ldots,\abs{x_d})
$$
for any $x_1,\ldots,x_d$. With $x_1=\ldots=x_d=1$, this gives
$$
E\left(\max(Z_1,\ldots,Z_d) \right) = 1 = E(Z_i), \ \forall i\in \{ 1,\ldots,d\}.
$$
It follows that for any $1\leq i\leq d$, $Z_i=\max(Z_1,\ldots,Z_d)=: X$ a.s., concluding the proof.
\end{proof}
Clearly, any $F$-norm on $\R^{d+1}$ induces a $D$-norm on $\R^{d+1}$ with the first element of the generator being the constant 1, in the sense that if $\bfX=(X_1,\ldots,X_d)$ generates an $F$-norm $\norm{\cdot}_F$, the quantity
$$
\norm{\left(x_0, x_1,\ldots, x_d \right)} := \norm{\left(x_0, \frac{x_1}{E(X_1)},\ldots, \frac{x_d}{E(X_d)} \right)}_F
$$
defines a $D$-norm generated by $(1,X_1/E(X_1),\ldots,X_d/E(X_d))$. In particular, if $E(X_i)=1$ for any $1\leq i\leq d$, any $F$-norm is also a $D$-norm.

There are however $D$-norms which are not $F$-norms. The $L^1-$norm $\norm\cdot_1$ is a prominent example: we know from Example~\ref{exL1} that it is not an $F$-norm, although it is generated by a random permutation of the vector $(d+1,0,\dots,0)\in\R^{d+1}$ and is therefore a $D$-norm. We can actually deduce this from the following stronger result. We omit its elementary proof.
\begin{prop}
\label{propbounds}
Let $\bfX=(X_1,\ldots,X_d)$ be a rv satisfying $(\mathcal{H})$ and $\norm{\cdot}_F$ be the corresponding $F$-norm. For any $\bfx\in \R^{d+1}$, we have the bounds
$$
\max(\abs{x_0},\abs{x_1}E(X_1),\ldots,\abs{x_d}E(X_d)) \leq \norm{\bfx}_F\leq \abs{x_0} + \sum_{i=1}^d \abs{x_i}E(X_i).
$$
The upper bound is always strict if both $x_0$ and at least one of the $x_i$ $(1\leq i\leq d)$ are nonzero.
\end{prop}
%
%
%
%
%
While the upper bound in Proposition~\ref{propbounds} is not an $F$-norm, the weighted sup-norm in the lower bound is, as we saw in Example~\ref{exdeg}. In the case $E(X_1) = \cdots = E(X_d) = 1$, this is just the standard sup-norm on $\R^{d+1}$; from Takahashi's characterization~\citep[see][Theorem 1.3.1]{falk2019}, we know that this norm is special within the class of $D$-norms, as it is completely characterized by its value at $(1,\dots,1)$:
$$
\norm{(1,\dots,1)}_D=1 \Leftrightarrow \norm\cdot_D=\norm\cdot_\infty.
$$
The following result gives a corresponding characterization, within the class of $F$-norms, for the weighted sup-norm appearing as the lower bound in Proposition~\ref{propbounds}.
\begin{prop}
\label{propcharinfty}
Let $\bfX=(X_1,\ldots,X_d)$ be a rv satisfying $(\mathcal{H})$ and $\norm{\cdot}_F$ be the corresponding $F$-norm. Define $c_i=E(X_i)$, for $1\leq i\leq d$, and introduce the weighted sup-norm
$$
\norm{(x_0,x_1,\ldots,x_d)}_{\infty,\bfc} := \max(\abs{x_0},\abs{x_1} c_1,\ldots,\abs{x_d} c_d).
$$
Then
$$
\norm{\left(1,\frac{1}{c_1},\ldots,\frac{1}{c_d} \right)}_F = 1 \Leftrightarrow \norm\cdot_F = \norm\cdot_{\infty,\bfc}.
$$
\end{prop}
\begin{proof} The norm
$$
\norm{(x_0,x_1,\ldots,x_d)}_{D} := E\left[\max\left(\abs{x_0}, \abs{x_1} \frac{X_1}{E(X_1)},\dots,\abs{x_d} \frac{X_d}{E(X_d)}\right) \right]
$$
is a $D$-norm on $\R^{d+1}$, and satisfies $\norm{(1,\dots,1)}_D=1$. By Takahashi's characterization, it follows that $\norm\cdot_{D} = \norm\cdot_{\infty}$. Conclude then by noting that
$$
\norm{(x_0,x_1,\ldots,x_d)}_F = \norm{(x_0,x_1 c_1,\ldots,x_d c_d)}_{D} = \norm{(x_0,x_1 c_1,\ldots,x_d c_d)}_{\infty}
$$
and $\norm{(x_0,x_1 c_1,\ldots,x_d c_d)}_{\infty} = \norm{(x_0,x_1,\ldots,x_d)}_{\infty,\bfc}$.
\end{proof}
Outside of these extreme cases, many $D$-norms are automatically $F$-norms. This is a consequence of the following result.

\begin{lemma}\label{lem:certain D-norms are F-norms}
Let $\norm\cdot_D$ be a $D$-norm on $\R^{d+1}$, with the additional property that it has a generator $\bfZ=(Z_0,\dots,Z_d)$ with $Z_i>0$, $0\le i\le d$. Then it also has a generator $\bfZ^*=(1,Z_1^*,\dots,Z_d^*)$.  
\end{lemma}
\begin{cor}
\label{cor:certain D-norms are F-norms}
Any $D$-norm on $\R^{d+1}$ having a componentwise positive generator is also an $F$-norm on $\R^{d+1}$.
\end{cor}
A consequence of Corollary~\ref{cor:certain D-norms are F-norms} is that the $L^1$-norm $\norm\cdot_1$ on $\R^{d+1}$, which is not an $F$-norm, cannot have a $D$-norm generator $\bfZ=(Z_0,\dots,Z_d)$ with $Z_i>0$ for $0\le i\le d$. Another consequence is that the $L^p$-norm on $\R^{d+1}$, for $1<p<\infty$, is always an $F$-norm; see Proposition~1.2.1 in~\citet{falk2019}.
\begin{proof}[Proof of Lemma~\ref{lem:certain D-norms are F-norms}]
Let $S^*:=\set{(1,s_1,\dots,s_d):\,s_i\ge 0,\,1\le i\le d}$. The set $S^*$ is an {\it angular set}, in the sense that each $\bfx=(x_0,\dots,x_d)\in(0,\infty)^{d+1}$ can be represented as
\[
\bfx=x_0\left(1,\frac{x_1}{x_0},\dots,\frac{x_d}{x_0} \right)=: r\bfs,
\]
with $\bfs\in S^*$ and $r>0$ being uniquely determined. The radial function $R(\bfx):=x_0$ is positively homogeneous of order one. The assertion now follows by repeating the arguments in the derivation of the normed generators theorem in \citet[Theorem 1.7.1]{falk2019}.
\end{proof}
We close this section by highlighting an interesting connection between $F$-norms and $D$-norms in the context of multivariate extreme value theory. Recall that a multivariate df $F$ is said to belong to the domain of attraction of a multivariate max-stable distribution $G$ if there are sequences $(\bfa_n)$, $(\bfb_n)$, $\bfa_n>\bfzero$, with
$$
\forall \bfx\in\R^d, \ F^n(\bfa_n \bfx + \bfb_n) \to G(\bfx), \ \mbox{ as } n \to\infty.
$$
It also follows from a theorem of~\citet{skl59} that $F$ can be written
$$
F(\bfx) = C(F_1(x_1),\ldots,F_d(x_d))
$$
where $C$ is a copula function on $[0,1]^d$ ({\it i.e.} a df with standard uniform margins) and $F_i$ is the $i$th marginal distribution of $F$. By results of~\citet{deh84} and~\citet{gal87}, the above convergence is true if and only if it is true for the univariate margins of $F$, together with the following asymptotic expansion on the copula function $C$:
$$
C(\bfu) = 1-\norm{\bfone-\bfu}_D+\operatorname{o}(\norm{1-\bfu}_D) \ \mbox{ as } \bfu\uparrow \bfone.
$$
Here $\norm{\cdot}_D$ is a $D$-norm on $\R^d$ which describes the dependence structure in the limiting max-stable df $G$; for instance, if $G$ is standardized to have negative exponential margins, then $G(\bfx) = \exp(-\norm{\bfx}_D)$, $\bfx\leq \bfzero$. Of prime interest is the extremal coefficient $\norm{\bfone}_D$, which characterizes asymptotic dependence within the copula $C$:
\begin{itemize}
\item If $\norm{\bfone}_D = 1$, corresponding to $\norm{\cdot}_D = \norm{\cdot}_{\infty}$, then there is complete asymptotic dependence,
\item If $\norm{\bfone}_D = d$, corresponding to $\norm{\cdot}_D = \norm{\cdot}_1$, then there is asymptotic independence.
\end{itemize}
If $C$ is such a copula then it is the df of a vector with uniform marginal distributions, and thus one can naturally consider the $F$-norm it generates. The final result of this section shows that $\norm{\bfone}_D$ can be retrieved from the knowledge of this $F$-norm.
\begin{prop}
\label{linkDnormFnorm}
Let $C$ be a copula function on $[0,1]^d$ such that
$$
C(\bfu) = 1-\norm{\bfone-\bfu}_D+\operatorname{o}(\norm{\bfone-\bfu}_D) \ \mbox{ as } \bfu\uparrow \bfone
$$
where $\norm{\cdot}_D$ is some $D$-norm on $\R^d$. If $\norm{\cdot}_C$ is the $F$-norm on $\R^{d+1}$ corresponding to the copula $C$ then
$$
1-\norm{(x,1,\ldots,1)}_C = (1-x) - \frac{(1-x)^2}{2} \norm{\bfone}_D + \operatorname{o}((1-x)^2) \ \mbox{ as } x\uparrow 1.
$$
\end{prop}
\begin{proof} By Lemma~\ref{lemfund},
$$
\norm{(x,1,\ldots,1)}_C = x + \int_{x}^{\infty} [1-C(t \bfone)] \, dt= x + \int_x^1 [1-C(t \bfone)] \, dt.
$$
Using the dominated convergence theorem, this yields
\begin{align*}
\norm{(x,1,\ldots,1)}_C &= x + \int_x^1 [\norm{\bfone-t\bfone}_D+\operatorname{o}(\norm{\bfone-t\bfone}_D)] \, dt \\
						&= x + \norm{\bfone}_D \int_x^1 (1-t) \, dt + \operatorname{o}((1-x)^2) \ \mbox{ as } \ x\uparrow 1.
\end{align*}
Rearranging concludes the proof.
\end{proof}
Such a result opens the door to estimation procedures of the extremal coefficient $\norm{\bfone}_D$ based on estimation of $F$-norms. We deal more generally with convergence and sample versions of $F$-norms in the next section.
\section{Limiting behavior and estimation of $F$-norms}
\label{seclim}

Although the pointwise limit of a convergent sequence of $D$-norms is again a $D$-norm~\citep[see][Corollary 1.8.5]{falk2019}, this is no longer true for $F$-norms: for instance, if $(p_n)$ is a sequence of real numbers with $p_n>1$ and $p_n\downarrow 1$, then $\norm\cdot_{p_n}\to\norm\cdot_1$, and $\norm\cdot_{p_n}$ is for each $n$ an $F$-norm, but the limit $\norm\cdot_1$ is not.

However, if we ask that the limit is an $F$-norm, then we can relate the convergence of $F$-norms with convergence of distributions in the {\it Wasserstein metric}. Recall that the Wasserstein metric\index{Metric!Wasserstein} between two probability
distributions $P,Q$ on $\R^d$ with finite first moments in each component is
\begin{eqnarray*}
 & & d_W(P,Q) \\
 &:=& \inf\{E(\norm{\bfX-\bfY}_1)\!:\, \bfX\mathrm{\ has\ distribution\ }P, \ \bfY \mathrm{\ has\ distribution\ }Q\}.
\end{eqnarray*}
Convergence of probability measures $P_n$ to $P$ on $\R^d$ with respect to the
Wasserstein metric is equivalent to weak convergence together with
convergence of the moments
\[
\int_{\R^d} \norm{\bfx}_1\,P_n(d \bfx) \to \int_{\R^d}\norm{\bfx}_1\,P(d \bfx);
\]
see {\it e.g.}~\citet[][Definition 6.8 and Theorem 6.9]{villani09}. With this definition in mind, we can show the following result.
\begin{theorem}
\label{theoWass}
Pointwise convergence of a sequence of $F$-norms $\norm\cdot_{F_n}$ to an $F$-norm $\norm\cdot_{F}$ is equivalent to convergence of the sequence of distributions $F_n$ to $F$ in the Wasserstein metric.
\end{theorem}
\begin{proof} Pointwise convergence of $\norm\cdot_{F_n}$ to $\norm\cdot_{F}$ implies pointwise convergence of the sequence of max-CFs of $F_n$ (as defined in~\eqref{defn:definition of max-CF}) to the max-CF of $F$, which entails the desired convergence in the Wasserstein metric by~Theorem 2.1 in~\citet{falkst16}.

Conversely, if $F_n\to F$ in the Wasserstein metric, let $\bfX^{(n)}$ and $\bfX$ have dfs $F_n$ and $F$. For any $\bfx = (x_0,x_1,\ldots,x_d)\geq \bfzero$,
\begin{eqnarray*}
 & & \max(x_0,x_1 X_1^{(n)},\ldots,x_d X_d^{(n)}) \\
 &=& \max(x_0,x_1 [X_1 + (X_1^{(n)}-X_1)],\ldots,x_d [X_d + (X_d^{(n)}-X_d)]) \\
 &\leq & \max(x_0,x_1 X_1,\ldots,x_d X_d) + \max_{1\leq i\leq d} x_i |X_i^{(n)}-X_i|.
\end{eqnarray*}
An analogue inequality holds if we switch $\bfX^{(n)}$ and $\bfX$. We can then integrate to find
$$
| \norm{\bfx}_{F_n} - \norm{\bfx}_F | \leq \norm{\bfx}_{\infty} E\left( \norm{\bfX^{(n)} - \bfX}_1 \right).
$$
Since $\bfX^{(n)}$ and $\bfX$ were arbitrary rvs having dfs $F_n$ and $F$, this yields
\begin{equation}
\label{FnormLip}
| \norm{\bfx}_{F_n} - \norm{\bfx}_F | \leq \norm{\bfx}_{\infty} d_W(F_n,F)\to 0
\end{equation}
which concludes the proof.
\end{proof}
\noindent
Based on this result, as well as on our examples in Section~\ref{defi}, we can illustrate how the concept of $F$-norms can be used to prove convergence theorems. The following corollary focuses on the class of Pareto distributions and is an immediate consequence of Example~\ref{expareto} and Theorem~\ref{theoWass}. 
\begin{cor}
\label{corpareto}
Let $(\gamma_n)$ be a real-valued sequence with $0<\gamma_n<1$ for each $n$ and $\gamma_n\to \gamma\in (0,1)$. Let also, for each $n$, $F_n$ be the Pareto distribution with tail index $\gamma_n$, and $F$ be the Pareto distribution with tail index $\gamma$. Then $(F_n)$ converges to $F$ in the Wasserstein metric. 
\end{cor}
\noindent
The use of $F$-norms makes it possible to prove convergence in distribution and of moments with a single calculation and thus obtain results such as Corollary~\ref{corpareto} with a concise proof. Of course, one could alternatively prove Corollary~\ref{corpareto} by proving separately the convergence of dfs and convergence of moments, but this requires two distinct calculations. Let us also note that while the $F$-norm of a Pareto distribution is easy to obtain and has a relatively simple expression, its standard characteristic function ({\it i.e.} Fourier transform) is more involved and depends on the Gamma function evaluated in the complex plane.

The nice behavior of $F$-norms with respect to sequences of distributions naturally raises the question of what happens when $F_n$ is chosen to be the empirical df based on iid copies $\bfX^{(1)},\dots,\bfX^{(n)}$ of a rv $\bfX$ satisfying~$(\mathcal{H})$, {\it i.e.}
\[
\widehat{F}_n(\bft):= \frac{1}{n} \sum_{i=1}^n \indic_{\{ \bfX^{(i)} \leq \bft \}}, \ \bft\in\R^d.
\]
The (random) $F$-norm generated by $\widehat{F}_n$ is nothing but
$$
\norm{\bfx}_{\widehat{F}_n} = \frac{1}{n} \sum_{i=1}^n\max\left(\abs{x_0},\abs{x_1}X_1^{(i)},\dots,\abs{x_d}X_d^{(i)}\right).
$$
The law of large numbers then implies, for each $\bfx\in\R^{d+1}$, that a.s.
$$
\norm{\bfx}_{\widehat{F}_n} \to E\left(\max\left(\abs{x_0},\abs{x_1}X_1,\dots,\abs{x_d}X_d\right) \right) = \norm{\bfx}_F \ \mbox{ as } \ n\to\infty.
$$
This convergence suggests that the estimation of an $F$-norm is completely straightforward; by contrast, estimating the related concept of a $D$-norm in the context of multivariate extreme value analysis requires quite sophisticated techniques.

We now provide further insight into the convergence of $\norm{\cdot}_{\widehat{F}_n}$ to $\norm{\cdot}_F$. Noting that for any $\bfx$ in a box $K=\prod_{i=0}^d [a_i,b_i] \subset [0,\infty)^{d+1}$ we have, by monotonicity of $F$-norms,
\begin{eqnarray*}
\norm{\bfx}_{\widehat{F}_n}-\norm{\bfx}_F &\leq & \left( \norm{\bfb}_{\widehat{F}_n}-\norm{\bfb}_F \right) + \left( \norm{\bfb}_F-\norm{\bfa}_F \right) \\
\mbox{and } \ \norm{\bfx}_F-\norm{\bfx}_{\widehat{F}_n} &\leq & \left( \norm{\bfa}_F-\norm{\bfa}_{\widehat{F}_n} \right) + \left( \norm{\bfb}_F-\norm{\bfa}_F \right),
\end{eqnarray*}
the following locally uniform refinement of the pointwise almost sure convergence of $\norm{\cdot}_{\widehat{F}_n}$ to $\norm{\cdot}_F$ is a direct consequence of the continuity of $\norm\cdot_F$.
 
\begin{theorem}
\label{lem:uniform convergence on compact intervals}
Let $\bfX^{(1)},\dots,\bfX^{(n)}$ be iid copies of a rv $\bfX$ satisfying~$(\mathcal{H})$, with df $F$. Let $\norm{\cdot}_{\widehat{F}_n}$ be the random $F$-norm generated by the empirical df $\widehat{F}_n$ of this sample. We then have, for any $\bfx_0\geq\bfzero$,
\[
\sup_{\bfzero\le\bfx\le\bfx_0}\abs{\norm{\bfx}_{\widehat{F}_n}-\norm{\bfx}_F} \to 0 \ \mbox{ a.s.}
\]
\end{theorem}

To analyse the rate of (uniform) convergence of $\norm{\cdot}_{\widehat{F}_n}$ to $\norm{\cdot}_F$, we define the {\it empirical $F$-norm process}
\[
S_n= (S_n(\bfx))_{\bfx\ge\bfzero}:= \sqrt{n} \left(\norm{\bfx}_{\widehat{F}_n}-\norm{\bfx}_F \right)_{\bfx\ge\bfzero}
\]
on $[0,\infty)^{d+1}$. This stochastic process has continuous sample paths and satisfies $S_n(\bfzero)=0$. Suppose then that $E(X_i^2)<\infty$ for any $i\in\{1,\ldots,d\}$. Based on the standard central limit theorem, which gives the pointwise asymptotic normality of $S_n$, we may ask the question of the limiting behavior of the process $S_n$. For ease of exposition, we state a result in the case $d=1$.

\begin{theorem}
\label{theoCLT1}
Let $X^{(1)},\dots,X^{(n)}$ be iid copies of a univariate rv $X$ with df $F$. Assume that $X$ is nonnegative, with nonzero expectation and finite variance. Let $\norm{\cdot}_{\widehat{F}_n}$ be the random $F$-norm generated by the empirical df $\widehat{F}_n$ of this sample. For any $x_0, y_0>0$, we have
\[
S_n(x,y) := \sqrt{n} \left(\norm{(x,y)}_{\widehat{F}_n}-\norm{(x,y)}_F \right) \to  S(x,y)
\]
weakly in the space $C([0,x_0]\times [0,y_0])$ of continuous functions over $[0,x_0]\times [0,y_0]$, where the limiting process $S$, which should be read as 0 when $y=0$, is a bivariate Gaussian process with covariance structure
\begin{align*}
 & \operatorname{Cov}( S(x_1,y_1), S(x_2,y_2) ) \\
 &= x_1 x_2 \iint_{[1,\infty)^2} \left[ F\left( \min\left\{ \frac{x_1}{y_1} u, \frac{x_2}{y_2} v \right\} \right) - F\left( \frac{x_1}{y_1} u \right) F\left( \frac{x_2}{y_2} v \right) \right] du \, dv.
\end{align*}
\end{theorem}
Under the further assumption that $\int_0^{\infty} \sqrt{F(u) [1-F(u)]} \, du <\infty$~\citep[which is equivalent to $E(X^2)<\infty$ when $F$ is regularly varying at infinity, according to {\it e.g.}][p.276]{ser1980} we have the representation
$$
S(x,y) \stackrel{d}{=} y\int_{x/y}^{\infty} W\circ F(u) \, du
$$
as processes in $C([0,x_0]\times [0,y_0])$, where $W$ is a Brownian bridge on $[0,1]$. Indeed, since for any $t\in [0,1]$ the rv $W(t)$ is Gaussian centered with variance $t(1-t)$, we have $\mathbb{E}|W(t)| = \sqrt{t(1-t)}\sqrt{2/\pi}$, and thus
\begin{eqnarray*}
\mathbb{E}\left|\int_{x/y}^{\infty} W\circ F(u) \, du \right| &\leq & \int_0^{\infty} \mathbb{E} |W\circ F(u)| \, du \\
	&=& \sqrt{\frac{2}{\pi}} \int_0^{\infty} \sqrt{F(u) [1-F(u)]} \, du <\infty
\end{eqnarray*}
so that $y\int_{x/y}^{\infty} W\circ F(u) \, du$ is well-defined and a.s. finite. It is then straightforward to show, using the covariance properties of $W$, that the covariance structure of this Gaussian process coincides with that of $S$.
\begin{proof} The random functions $S_n$ and $S$ are elements of the functional space $C([0,x_0]\times [0,y_0])$. By Theorem~7.5 in~\citet{bil1999}, it suffices to show the convergence of finite-dimensional margins of $S_n$ to those of $S$ along with tightness of $(S_n)$, in the sense of tightness of its sequence of distributions.

We start by convergence of finite-dimensional margins. The multivariate central limit theorem implies, for nonnegative pairs $(x_1,y_1), \ldots, (x_k,y_k)$, that the rv $(S_n(x_1,y_1),\ldots, S_n(x_k,y_k))$ converges weakly to a centered Gaussian distribution. By Hoeffding's identity~\citep[see][Lemma 2.5.2]{falk2019}, the limiting covariance matrix is described by
\begin{align*}
 &  \operatorname{Cov}( \max(x_i, y_i X), \max(x_j, y_j X) ) \\
 &= \iint_{\R^2} [ P(\max(x_i, y_i X) \leq x, \max(x_j, y_j X)\leq y) \\
 & \hspace*{1.3cm} - P(\max(x_i, y_i X) \leq x) P(\max(x_j, y_j X)\leq y)] \, dx \, dy \\
 &= \int_{x_i}^{\infty} \int_{x_j}^{\infty} [ P( y_i X \leq x, y_j X\leq y) - P(y_i X \leq x) P(y_j X \leq y)] \, dx \, dy.
\end{align*}
This is clearly equal to 0 when either $y_i$ or $y_j$ is 0, and otherwise, using the change of variables $x= x_i u, y= x_j v$, we find
\begin{align*}
 &  \operatorname{Cov}( \max(x_i, y_i X), \max(x_j, y_j X) ) \\
 &= x_i x_j \iint_{[1,\infty)^2} \left[ F\left( \min\left\{ \frac{x_i}{y_i} u, \frac{x_j}{y_j} v \right\} \right) - F\left( \frac{x_i}{y_i} u \right) F\left( \frac{x_j}{y_j} v \right) \right] du \, dv
\end{align*}
which is exactly the covariance structure of the Gaussian process $S$.

We now show tightness, that is, for any $\varepsilon>0$,
$$
\lim_{\delta\to 0} \limsup_{n\to\infty} P\left( \sup_{\substack{(x_1,y_1),(x_2,y_2)\in  [0,x_0]\times [0,y_0] \\ \max(|x_1-x_2|,|y_1-y_2|)\leq \delta}} |S_n(x_1,y_1) - S_n(x_2,y_2)| >\varepsilon\right) = 0,
$$
or, in other words, that $S_n$ is stochastically equicontinuous on $[0,x_0]\times [0,y_0]$. The key to the proof is threefold. Firstly, we apply Theorem~1 p.93 of~\citet{showel1986} to construct, on a common probability space, a triangular array $(U^{(n,1)},\ldots,U^{(n,n)})_{n\geq 1}$ of rowwise independent, standard uniform rvs, and a Brownian bridge $\widetilde{W}$ such that
$$
\sup_{0\leq t\leq 1} |\mathbb{W}_n(t) - \widetilde{W}(t)| \to 0 \ \mbox{ a.s. with } \ \mathbb{W}_n(t) := \frac{1}{\sqrt{n}} \sum_{i=1}^n \left[ \indic_{\{ U^{(n,i)} \leq t \}} - t\right].
$$
Secondly, if we denote by $q$ the quantile function of $X$ ({\it i.e.} the left-continuous inverse of $F$) and by $\widetilde{X}^{(n,i)} := q(U^{(n,i)})$, we have, for any $n\geq 1$,
$$
S_n(x,y) \stackrel{d}{=} \widetilde{S}_n(x,y) := \frac{1}{\sqrt{n}} \sum_{i=1}^n \left[ \max\left(x,y \, \widetilde{X}^{(n,i)}\right) - E(\max(x,y X)) \right],
$$
as processes in $C([0,x_0]\times [0,y_0])$. We may and will therefore prove our result using $\widetilde{S}_n$ rather than $S_n$. Thirdly and finally, if $(x,y)\in [0,x_0]\times [0,y_0]$, we have
$$
\min\left(x,y \widetilde{X}^{(n,i)}\right) = y\int_0^{x/y} \left[ 1-\indic_{\{ \widetilde{X}^{(n,i)} \leq u  \}} \right] \, du.
$$
Since $\widetilde{X}^{(n,i)} \leq u\Leftrightarrow U^{(n,i)}\leq F(u)$, this yields
$$
\frac{1}{\sqrt{n}} \sum_{i=1}^n \left[ \min\left(x,y \widetilde{X}^{(n,i)}\right) - E(\min\left(x,y X\right)) \right] = -y\int_0^{x/y} \mathbb{W}_n \circ F(u) \, du.
$$
Using the identity $\max(a,b)=a+b-\min(a,b)$, valid for any $a,b\geq 0$, it follows that:
\begin{align*}
\widetilde{S}_n(x_1,y_1) - \widetilde{S}_n(x_2,y_2) &= (y_1 - y_2) \times \frac{1}{\sqrt{n}}\sum_{i=1}^n [\widetilde{X}^{(n,i)} - E(X)] \\
							&+ T_n(x_1,y_1) - T_n(x_2,y_2) \\
\mbox{with } T_n(x,y)		&:= y\int_0^{x/y} \mathbb{W}_n \circ F(u) \, du.
\end{align*}
The first term on the right-hand side above is stochastically equicontinuous, because the random term is a $\operatorname{O}_{\mathbb{P}}(1)$ (by the Chebyshev inequality). We conclude the proof by focusing on $T_n(x,y)$, and for this we first remark that
$$
\sup_{\substack{0\leq x\leq x_0 \\ 0\leq y\leq y_0}} \left| y\int_0^{x/y} \left[ \mathbb{W}_n \circ F(u) - \widetilde{W}\circ F(u) \right] du \right| \leq x_0 \sup_{0\leq t\leq 1} |\mathbb{W}_n(t) - \widetilde{W}(t)| \to 0
$$
almost surely. A consequence of this convergence is that, to show the stochastic equicontinuity of $T_n$, it is enough to prove that the random function $T$ defined by
$$
T(x,y) := \begin{cases} y\displaystyle\int_0^{x/y} \widetilde{W}\circ F(u) \, du & \mbox{if } y>0, \\ 0 & \mbox{if } y=0, \end{cases}
$$
satisfies
$$
\lim_{\delta\to 0} P\left( \sup_{\substack{(x_1,y_1),(x_2,y_2)\in  [0,x_0]\times [0,y_0] \\ \max(|x_1-x_2|,|y_1-y_2|)\leq \delta}} |T(x_1,y_1) - T(x_2,y_2)| >\varepsilon\right) = 0.
$$
Recall that $\widetilde{W}$ has almost surely continuous sample paths on $[0,1]$, and thus $T$ is almost surely continuous on $[0,x_0]\times (0,y_0]$. Because, for any $y>0$,
$$
T(x,y) = \int_0^{x} \widetilde{W}\circ F(u/y) \, du
$$
and $T(x,0)=0$, it follows by the dominated convergence theorem that almost sure continuity of $T$ also holds on the {\it compact} set $[0,x_0]\times [0,y_0]$. Then $T$ must also be almost surely uniformly continuous on this set, and therefore
$$
\lim_{\delta\to 0}\sup_{\substack{(x_1,y_1),(x_2,y_2)\in  [0,x_0]\times [0,y_0] \\ \max(|x_1-x_2|,|y_1-y_2|)\leq \delta}} |T(x_1,y_1) - T(x_2,y_2)| = 0 \ \mbox{ a.s. }
$$
This completes the proof.
\end{proof}
In the case $d>1$, and under regularity conditions~\citep[{\it e.g.} those of][]{mas1989}, a similar proof using a special construction of the multivariate empirical process can be written to show an analogue of Theorem~\ref{theoCLT1}, which gives the convergence of the process $S_n$, in a space of continuous functions over compact subsets of $[0,\infty)^{d+1}$, to a $(d+1)-$dimensional Gaussian process $S$ with covariance structure
\begin{align*}
 & \operatorname{Cov}( S(x_1,\bfx_1), S(x_2,\bfx_2) ) \\
 &= x_1 x_2 \iint_{[1,\infty)^{2d}} \left[ F\left( \min\left\{ \frac{x_1}{\bfx_1} \bfu, \frac{x_2}{\bfx_2} \bfv \right\} \right) - F\left( \frac{x_1}{\bfx_1} \bfu \right) F\left( \frac{x_2}{\bfx_2} \bfv \right) \right] d\bfu \, d\bfv.
\end{align*}
Our objective is now to dwell upon the nice sequential behavior of $F$-norms and show an example of how this could be used to prove powerful theorems on the convergence of certain sequences of rvs. To this end we first need to understand better how to manipulate $F$-norms, which leads us to exploring their algebraic properties.
\section{Algebra of the set of $F$-norms}
\label{sec:algebra}

One can multiply $F$-norms $\norm\cdot_{F}$ and $\norm\cdot_{G}$ by constructing the $F$-norm generated by the componentwise product of any pair of {\it independent} rvs having dfs $F$ and $G$; independence is used to ensure that the distribution of this componentwise product is well-defined, and thus so is the product $F$-norm. We denote this operation by $\norm\cdot_{F} * \norm\cdot_{G}$. It coincides with taking products of $D$-norms if $\norm\cdot_{F}$ and $\norm\cdot_{G}$ have components with expectation 1, see~\citet[Section~1.9]{falk2019}. 
\begin{exam}[Product of Bernoulli $F$-norms]
\upshape The product of two independent Bernoulli rvs $X$ and $Y$ with respective parameters $p$ and $q$ is also a Bernoulli rv, with parameter $pq$. As a consequence, following Example~\ref{exber}, the resulting product $F$-norm is 
$$
\norm{(x_0,x_1)}_F = (1-pq) |x_0| + pq\max(|x_0|,|x_1|).
$$
\end{exam}
The previous example was easy to analyse because the product of two independent Bernoulli rvs is also a Bernoulli rv. In general cases, where the product of the two rv may not have such a simple distribution, the product $F$-norm can be calculated using the following Tonelli formula. 
\begin{prop}
\label{propprodFnorm}
Let $F$ and $G$ be the dfs of two rvs satisfying condition~$(\mathcal{H})$. Then, for any $\bfx=(x_0,x_1,\ldots,x_d)\in \R^{d+1}$,  
\begin{eqnarray*}
(\norm\cdot_{F} * \norm\cdot_{G}) (\bfx) &=& \int_{[0,\infty)^d} \norm{(x_0,x_1 t_1,\ldots,x_d t_d)}_{F} \, dG(t_1,\ldots,t_d) \\
													   &=& \int_{[0,\infty)^d} \norm{(x_0,x_1 t_1,\ldots,x_d t_d)}_{G} \, dF(t_1,\ldots,t_d).
\end{eqnarray*}
\end{prop}
\begin{proof} Let $\bfX = (X_1,\ldots,X_d)$ and $\bfY = (Y_1,\ldots,Y_d)$ be independent and have dfs $F$ and $G$. We have 
$$
(\norm\cdot_{F} * \norm\cdot_{G}) (\bfx) = E\left(\max\left(\abs{x_0}, \abs{x_1} X_1 Y_1,\dots,\abs{x_d} X_d Y_d \right)\right).
$$
By nonnegativity of $\max\left(\abs{x_0}, \abs{x_1} X_1 Y_1,\dots,\abs{x_d} X_d Y_d \right)$ and independence of $\bfX$ and $\bfY$, we find, using the Tonelli theorem, that  
\begin{eqnarray*}
 & & (\norm\cdot_{F} * \norm\cdot_{G}) (\bfx) \\ 
 &=& \int_{[0,\infty)^d} E\left(\max\left(\abs{x_0}, \abs{x_1} t_1 X_1,\dots,\abs{x_d} t_d X_d \right)\right) dG(t_1,\ldots,t_d)
\end{eqnarray*}
which is exactly the first formula. The second expression follows by swapping integration with respect to $dG$ for integration with respect to $dF$.
\end{proof}
\begin{exam}[Product of uniform $F$-norms]
\upshape Following Example~\ref{exunif}, the product of the standard uniform $F$-norm by itself has the expression  
\begin{eqnarray*}
(\norm{\cdot}_F * \norm{\cdot}_F)(x_0,x_1) &=& \int_0^1 \left( |x_0| \indic_{\{ t|x_1| \le |x_0|\}} + \dfrac{x_0^2+t^2 x_1^2}{2t|x_1|} \indic_{\{ t|x_1|>|x_0| \}} \right) dt \\
										   &=& \begin{cases} |x_0| & \mbox{if } |x_1| \leq |x_0|, \\[5pt] \displaystyle \frac{5}{4} |x_1| + \frac{x_0^2}{2|x_1|} \left[ \log\left( \frac{|x_1|}{|x_0|} \right) - \frac{1}{2}  \right] & \mbox{if } |x_1|> |x_0|. \end{cases}
\end{eqnarray*}
\end{exam}
Let us now explore in more detail the structure of the set of $F$-norms equipped with its multiplication. It is clear that the sup-norm $\norm\cdot_{\infty}$ on $\R^{d+1}$, with generator $(1,\ldots,1) \in \R^d$, is an identity element for this operation. It is also straightforward to see that it is the unique such element: if $\norm\cdot_{F}$ is an identity element for $*$ then
$
\norm\cdot_{F} = \norm\cdot_{F} * \norm\cdot_{\infty} = \norm\cdot_{\infty}.
$
We summarize this short discussion by the following result.
\begin{prop}
\label{propgroup}
The set of $F$-norms is a commutative monoid for the $F$-norm multiplication $*$, with identity element $\norm\cdot_{\infty}$. The only invertible elements are the $F$-norms generated by nonrandom vectors.
\end{prop}
\noindent
The only point we need to show in Proposition~\ref{propgroup} is the assertion about invertible elements. The key is to note the following lemmas.
\begin{lemma}
\label{lemchar}
Let $Z$ be a real-valued rv such that $|E(e^{itZ})|=1$ for any $t\in \R$. Then $Z$ is almost surely constant.
\end{lemma}
\begin{proof}[Proof of Lemma~\ref{lemchar}] We use the Cauchy-Schwarz inequality for the inner product $(X,Y)\mapsto E(X\overline{Y})$ on the space of complex-valued square-integrable rvs, to obtain:
$$
\forall t \in \R, \ |E(e^{itZ})|^2 = |E(e^{itZ}\cdot 1)|^2\leq 1.
$$
By assumption, we actually have equality here. This means that for any $t$, the rvs $e^{itZ}$ and 1 are almost surely proportional, {\it i.e.} $e^{itZ} = \lambda(t)$, with $\lambda(t) \in \mathbb{C}$. Define now the event $E_t:= \{ e^{itZ} = \lambda(t) \}$, and let $(t_n)$, $(t'_n)$ be two sequences converging to 0. Define $E=(\bigcap_n E_{t_n}) \cap (\bigcap_n E_{t'_n})$. Then $P(E)=1$ and on $E$,
$$
\frac{\lambda(t_n) - \lambda(0)}{t_n} = \frac{e^{it_n Z} - 1}{t_n} \to iZ \ \mbox{ and } \ \frac{\lambda(t'_n) - \lambda(0)}{t'_n} = \frac{e^{it'_n Z} - 1}{t'_n} \to iZ.
$$
It follows that the limit
$$
\lim_{n\to\infty} \frac{\lambda(t_n) - \lambda(0)}{t_n}
$$
exists and does not depend on the choice of $t_n\to 0$: the function $\lambda$ is differentiable at 0. Conclude, by using $(t_n)$ again, that on the event $(\bigcap_n E_{t_n})$, $\lambda'(0) = iZ$ and thus $Z$ is almost surely the constant $-i\lambda'(0)$.
\end{proof}
\begin{lemma}
\label{leminv}
Let $X$ and $Y$ be two independent nonnegative rvs such that $XY=1$ almost surely. Then $X$ and $Y$ are almost surely positive constants.
\end{lemma}
\begin{proof}[Proof of Lemma~\ref{leminv}] Necessarily $P(X=0)=P(Y=0)=0$. Then by assumption $\log X + \log Y$ is a.s. zero. Denote by $\varphi_X(t):=E(e^{it\log X})$ and $\varphi_Y(t):=E(e^{it\log Y})$ the characteristic functions of $\log X$ and $\log Y$. This entails $\varphi_X(t)\cdot \varphi_Y(t) = 1$ for any $t\in \R$, by independence. Since any characteristic function has a modulus at most 1, we find $|\varphi_X(t)| = |\varphi_Y(t)| = 1$. Conclude by applying Lemma~\ref{lemchar}.
\end{proof}
\begin{proof}[Proof of Proposition~\ref{propgroup}] Let $\norm\cdot_{F}$ and $\norm\cdot_{G}$ satisfy $\norm\cdot_{F} * \norm\cdot_{G} = \norm\cdot_{\infty}$. Equivalently, there are independent rvs $(X_1,\ldots,X_d)$ with df $F$ and $(Y_1,\ldots,Y_d)$ with df $G$ such that for any $\bfx=(x_0,x_1,\dots,x_d)\in\R^{d+1}$,
$$
E(\max(\abs{x_0}, \abs{x_1}X_1 Y_1,\dots,\abs{x_d}X_d Y_d)) = E(\max(\abs{x_0}, \abs{x_1} \cdot 1,\dots,\abs{x_d} \cdot 1)).
$$
By Theorem~\ref{theo:F-norm characterizes distribution}, we find that each $X_i Y_i$ is a.s. constant equal to 1. Conclude by applying Lemma~\ref{leminv}.
\end{proof}
The same kind of argument can be used to identify the set of idempotent elements for the multiplication of $F$-norms.
\begin{prop}
\label{propidem}
The only idempotent element for multiplication of $F$-norms is the sup-norm $\norm\cdot_{\infty}$.
\end{prop}
The proof is again based on an auxiliary result for real-valued rvs.
\begin{lemma}
\label{lemidem}
Let $X$ and $Y$ be two independent nonnegative rvs having the same distribution and satisfying $XY\stackrel{d}{=}X$. Then $X=Y=1$ almost surely.
\end{lemma}
\begin{proof}[Proof of Lemma~\ref{lemidem}] The assumption is $P(XY\leq t)=P(X\leq t)$ for any $t$. Note that
\begin{eqnarray*}
P(X=0)=P(XY=0) &=& P(X=0)+P(Y=0)-P(X=Y=0) \\	
			   &=& P(X=0) [2-P(X=0)]
\end{eqnarray*}
so that $P(X=0)\in \{ 0,1\}$, and necessarily $P(X=0)=0$ since $E(X)>0$. Then by assumption $\log X + \log Y \stackrel{d}{=} \log X$. If $\varphi(t):=E(e^{it\log X})$ denotes the characteristic function of $\log X$, this entails $[\varphi(t)]^2 = \varphi(t)$ for any $t\in \R$, by independence. Thus, for any $t\in\R$, $\varphi(t)\in \{0,1\}$. Noting that $\varphi(0)=1$ and $\varphi$ is continuous entails that necessarily $\varphi\equiv 1$, since $\varphi(\R)$ must be a path-connected subset of $\{0,1\}$. As a consequence, $\log X=0$ almost surely, completing the proof.
\end{proof}

\begin{proof}[Proof of Proposition~\ref{propidem}] Let $\norm\cdot_{F}$ be an idempotent element for the multiplication of $F$-norms, with generator $(X_1,\ldots,X_d)$. Let $(Y_1,\ldots,Y_d)$ be an independent copy of this rv. Since $\norm\cdot_{F}$ is idempotent, we have, for any $\bfx=(x_0,x_1,\dots,x_d)\in\R^{d+1}$,
$$
E(\max(\abs{x_0}, \abs{x_1}X_1 Y_1,\dots,\abs{x_d}X_d Y_d)) = E(\max(\abs{x_0}, \abs{x_1}X_1,\dots,\abs{x_d}X_d)).
$$
By Theorem~\ref{theo:F-norm characterizes distribution}, we find that $X_i Y_i \stackrel{d}{=} X_i$, for each $i\in \{1,\ldots,d\}$. Conclude by applying Lemma~\ref{lemidem}.
\end{proof}
That $F$-norms can be multiplied in the way we have described here constitutes a motivation for our way of extending the notion of $F$-norms to not necessarily nonnegative rvs, which we describe in the next section.
\section{$F$-norms of general random vectors}
\label{sec:generalrv}

The concept of $F$-norms focuses on the distribution of an arbitrary multivariate rv with nonnegative and integrable components. Our purpose here is to show how we can also define, in a sensible way, a concept of $F$-norms for a rv whose components can attain negative values, under an integrability condition.

Let $\bfX=(X_1,\dots,X_d)$ be an arbitrary rv satisfying $E(\exp(X_i))<\infty$, $1\le i\le d$. Then $\bfY:=\exp(\bfX)=(\exp(X_1),\dots,\exp(X_d))$ generates an $F$-norm $\norm\cdot_{F}^{\textrm{exp}}$. As the function $x\mapsto\exp(x)$ is a bijection from the real line onto the interval $(0,\infty)$, the distribution of $\bfX$ is characterized by the $F$-norm $\norm\cdot_{F}^{\textrm{exp}}$, which we call a {\it log $F$-norm}.

\begin{exam}[Multivariate normal distribution]
\upshape Put $Z:= \exp(X-\sigma^2/2)$, where $X$ follows the univariate normal distribution $N(0,\sigma^2)$. The rv $Z$ is log-normal distributed with $E(Z)=1$. The log $F$-norm of $X-\sigma^2/2$ is then just a $D$-norm and equals, for $x,y>0$,
$$
\norm{(x,y)}_F^{\textrm{exp}} = E\left(\max(x,yZ)\right)=x \Phi\left(\frac{\sigma}{2} + \frac{\log(x/y)}{\sigma} \right) + y \Phi\left(\frac{\sigma}{2} + \frac{\log(y/x)}{\sigma} \right),
$$
which is the so-called {\it H\"{u}sler-Reiss} $D$-norm with parameter $\sigma^2>0$~\citep[see][]{falk2019}; by $\Phi$ we denote the df of the standard normal distribution on $\R$. As a consequence, the normal distribution $N(-\sigma^2/2,\sigma^2)$ of $X-\sigma^2/2$ is characterized by the norm $\norm{\cdot}_F^{\textrm{exp}}$.

More generally, the log $F$-norm of the normal distribution $N(\mu,\sigma^2)$ with arbitrary $\mu\in\R$ and $\sigma^2>0$ is, for $x,y>0$,
\begin{align*}
\norm{(x,y)}_F^{\textrm{exp}}&= E\left(\max\left(x,y\exp\left(\mu+\frac{\sigma^2}2\right)Z\right)\right)\\
&= x\Phi\left(\frac{\log(x/y)-\mu}\sigma\right) \! + y \exp\left(\mu+\frac{\sigma^2}2\right) \! \Phi\left(\sigma
 + \frac{\log(y/x)+\mu}\sigma\right).
\end{align*}
By Corollary~\ref{corclassif}, we should find back the log-normal df from this $F$-norm by differentiating $\norm{(t,1)}_F^{\textrm{exp}}$ on $(0,\infty)$. Clearly
\begin{align*}
\frac{d}{dt} (\norm{(t,1)}_F^{\textrm{exp}}) &= \Phi\left(\frac{\log(t)-\mu}\sigma\right) + \frac{1}{\sigma}\Phi'\left(\frac{\log(t)-\mu}\sigma\right) \\
							  &- \frac{1}{t\sigma}\exp\left(\mu+\frac{\sigma^2}2\right) \Phi'\left(\sigma-\frac{\log(t)-\mu}\sigma\right).
\end{align*}
Note also that
\begin{align*}
\Phi'\left(\sigma-\frac{\log(t)-\mu}\sigma\right) &= \frac{1}{\sqrt{2\pi}} \exp\left(-\frac{1}{2} \left[ \sigma - \frac{\log(t) - \mu}{\sigma} \right]^2 \right) \\
											      &= t \exp\left( - \mu - \frac{\sigma^2}2\right) \times \Phi'\left(\frac{\log(t)-\mu}\sigma\right)
\end{align*}
to find, as expected:
$$
\frac{d}{dt} (\norm{(t,1)}_F^{\textrm{exp}}) = \Phi\left(\frac{\log(t)-\mu}\sigma\right).
$$
\end{exam}
Combining the discussion we have developed in the previous example with Theorem~\ref{theoWass} leads, without any further calculation, to the following immediate result. This serves as a further example of how the asymptotic results in Section~\ref{seclim} may be used to establish asymptotic theory.
\begin{cor}
\label{corlognor}
Let $(\mu_n)$, $(\sigma_n)$ be real-valued sequences such that $\mu_n\to \mu$ and $\sigma_n\to\sigma>0$. Then:
\begin{itemize}
\item The sequence of log-normal distributions with parameters $\mu_n$ and $\sigma_n^2$ converges to the log-normal distribution with parameters $\mu$ and $\sigma^2$ in the Wasserstein metric.
\item The sequence $G_n$ of normal distributions with parameters $\mu_n$ and $\sigma_n^2$ converges in distribution to the normal distribution $G$ with parameters $\mu$ and $\sigma^2$, and the moments of $G_n$ converge to those of $G$.
\end{itemize}
\end{cor}
More generally, if $\bfX$ follows a multivariate normal distribution $N(\bfmu,\Sigma)$ with mean vector $\bfmu\in\R^d$ and covariance matrix $\Sigma= (\sigma_{ij})_{1\le i,j\le d}$, then each component $Y_i=\exp(X_i)$ is log-normal distributed with mean $E(Y_i)=\exp(\mu_i+\sigma_{ii}/2)$. In analogy to the $D$-norm generated by the normalized rv $\bfZ=\bfY/E(\bfY)$ and called a {\it H\"{u}sler-Reiss} $D$-norm~\citep[see][]{falk2019}, we call the $F$-norm corresponding to $\bfY$ a {\it H\"{u}sler-Reiss} $F$-norm. It characterizes the normal distribution $N(\bfmu,\Sigma)$.

The concept of log $F$-norms for rvs with an arbitrary sign is not adapted solely to Gaussian distributions, as we show in the following examples.
\begin{exam}[Gumbel distribution]
\upshape Let $X$ have the standard negative Gumbel distribution, {\it i.e.}
$$
P(X\leq t) = 1-e^{-e^t}, \ t\in \R.
$$
Then $\exp X$ has a unit exponential distribution, and therefore the log $F$-norm characterizing the standard negative Gumbel distribution is
$$
\norm{(x_0,x_1)}_F^{\textrm{exp}} =  |x_0| + |x_1|\exp\left(-\frac{|x_0|}{|x_1|} \right)
$$
when $x_1\neq 0$, and $|x_0|$ otherwise (see Example~\ref{exexpo}).
\end{exam}

\begin{exam}[On the central limit theorem]
\upshape Let $\bfX^{(1)},\bfX^{(2)},\dots$ be iid copies of a centered rv $\bfX=(X_1,\dots,X_d)$ having covariance matrix $\Sigma$, and a finite moment generating function in a neighborhood of the origin, {\it i.e.} there exists $\varepsilon>0$ with $\varphi_j(t):=E(\exp(tX_j))<\infty$ for any $\abs t< \varepsilon$ and $1\le j\le d$. The multivariate central limit theorem and continuous mapping theorem imply
\begin{equation}\label{eqn:convergence of log F-norms}
\exp\left(\frac 1{\sqrt{n}}\sum_{i=1}^n\bfX^{(i)}\right) \stackrel{d}{\longrightarrow} \exp(\bfxi),
\end{equation}
where $\bfxi=(\xi_1,\dots,\xi_d)$ follows a multivariate normal distribution with mean vector zero and covariance matrix $\Sigma$. Besides, we have 
$$
\mathbb{E}\left[ \exp\left(\frac{2}{\sqrt{n}}\sum_{i=1}^n\bfX^{(i)}_j\right) \right] = \mathbb{E}\left[ \prod_{i=1}^n \exp\left(\frac{2}{\sqrt{n}} \bfX^{(i)}_j \right) \right] = \left[ \varphi_j(2/\sqrt{n}) \right]^n.
$$
Since $X_j$ is centered with variance $\Sigma_{jj}$, we have by a Taylor expansion
$$
\mathbb{E}\left[ \exp\left(\frac{2}{\sqrt{n}}\sum_{i=1}^n\bfX^{(i)}_j\right) \right] = \left[ 1 + \frac{2}{n} \Sigma_{jj} + \operatorname{o}\left(\frac{1}{n}\right) \right]^n \to e^{2\Sigma_{jj}}. 
$$
It follows that the sequence
\[
\exp\left(\frac 1{\sqrt{n}}\sum_{i=1}^n\bfX^{(i)}_j\right), \ n\geq 1,
\]
has a bounded second moment and thus is \textit{uniformly integrable}~\citep[see {\it e.g.}][]{bil1999} for each $j=1,\dots,d$. This entails convergence of the sequence of its first moments and, combined with~\eqref{eqn:convergence of log F-norms} and Theorem~\ref{theoWass}, pointwise convergence of the generated $\log F$-norms, {\it i.e.}
\begin{align*}
&E\left(\max\left(x_0,x_1\exp\left(\frac 1{\sqrt{n}}\sum_{i=1}^n\bfX^{(i)}_1\right), \dots, x_d\exp\left(\frac 1{\sqrt{n}}\sum_{i=1}^n\bfX^{(i)}_d \right)\right)\right)\\
&\to E\left(\max\left(x_0,x_1\exp(\xi_1),\dots,x_d\exp(\xi_d)\right) \right) \ \mbox{ as } \ n\to\infty
\end{align*}
for each $x_0,x_1,\dots,x_d\ge 0$. We thus have a convergence of $F$-norms akin to the central limit theorem.
%
\end{exam}
We could, of course, have used in place of the exponential function any one-to-one increasing transformation from $\R$ to $(0,\infty)$ in order to define an $F$-norm for general rvs. Another potential transformation would have been
$$
x\mapsto \frac{\pi}{2} + \arctan x,
$$
which has the appeal of avoiding any integrability condition on the rv $\bfX$. The exponential function, however, interacts well with our notion of product of $F$-norms, in the sense that
$$
\norm{\cdot}_{\bfX}^{\textrm{exp}} * \norm{\cdot}_{\bfY}^{\textrm{exp}} = \norm{\cdot}_{\bfX + \bfY}^{\textrm{exp}}
$$
if $\bfX$ and $\bfY$ are independent: a product of two log $F$-norms is the log $F$-norm corresponding to the convolution of their individual distributions. 

\section{Geometry of $F$-norms}
\label{sec:geometry}

%
%
%

The original motivation for constructing $F$-norms was to combine the distributional properties of a max-CF with the structure of a $D$-norm into a single mathematical object. We have so far concentrated on the information that $F$-norms bring about multivariate distributions. We use here the geometry of the $F$-norms to find yet other different objects who summarize a multivariate distribution.

Since any norm $\norm\cdot$ on $\R^{d+1}$ is homogeneous, an immediate consequence is that each $F$-norm $\norm\cdot_F$ is uniquely determined by its values on the unit sphere for $\norm\cdot$, namely $S_{\norm{\cdot}}:=\set{\bfu\in\R^{d+1}:\,\norm{\bfu}=1}$: to put it differently, we have for $\bfx\in\R^{d+1}$, $\bfx\not=\bfzero$,
\begin{equation}
\norm{\bfx}_F=\norm{\bfx}\norm{\frac{\bfx}{\norm{\bfx}}}_F,
\end{equation}
with $\bfx/\norm{\bfx}\in S_{\norm{\cdot}}$. By choosing $\norm\cdot=\norm\cdot_1$ and using the radial symmetry of the $L^1$-norm and of $F$-norms, we find that we need only consider the values of $\norm\cdot_F$ on the part of the sphere $S_{\norm{\cdot}}$ contained in $[0,\infty)^{d+1}$. In other words, each df $F$ of a rv $\bfX$ satisfying~$(\mathcal{H})$ is characterized by the function
\[
A(\bft) :=\norm{\left(1-\sum_{i=1}^d t_i, t_1,\dots,t_d \right)}_F, \ \bft=(t_1,\dots,t_d),
\]
defined on $\Delta_1:=\set{\bft\in[0,1]^d:\,\sum_{i=1}^d t_i\le 1}$. This construction is similar to that of the {\it Pickands dependence function} in multivariate extreme value theory~\citep[see {\it e.g.}][]{gudseg2010}, and we therefore call the function $A$ the Pickands dependence function of the $F$-norm $\norm{\cdot}_F$.
Let us briefly mention here that, based on a sample of copies of $\bfX$, we can estimate this Pickands dependence function by an empirical version, just as we did in Section~\ref{seclim} for the full $F$-norm: let $\bfX^{(1)},\dots,\bfX^{(n)}$ be iid copies of a rv $\bfX$ satisfying~$(\mathcal{H})$. Put, for $\bft\in\Delta_1$, with $t_0:=1-\sum_{i=1}^d t_i$,
\[
\widehat{A}_n(\bft):= \frac{1}{n} \sum_{i=1}^n \max\left(t_0,t_1X_1^{(i)},\dots,t_dX_d^{(i)} \right),
\]
which is that (random) Pickands dependence function which characterizes the empirical df $\widehat{F}_n$. The asymptotic properties of $\widehat{A}_n$ follow directly from our asymptotic results in Section~\ref{seclim}: since $\Delta_1$ is compact, we get, by Theorem~\ref{lem:uniform convergence on compact intervals},
\[
\sup_{\bft\in\Delta_1}\abs{\widehat{A}_n(\bft)-A(\bft)} \to 0 \ \mbox{ a.s.},
\]
and, by the multivariate extension of Theorem~\ref{theoCLT1} mentioned at the end of Section~\ref{seclim}, we have
\[
\sqrt{n} \left(\widehat{A}_n(\bft)-A(\bft)\right) \to \mathcal{S}(\bft)
\]
weakly in the space of the continuous functions on the unit simplex in $\R^{d+1}$, where $\mathcal{S}$ is a Gaussian process.


We now explore how, instead of characterizing an $F$-norm by a function such as its Pickands dependence function, we can identify it by a compact set which summarizes the geometry of an $F$-norm. Recall that an $F$-norm is characterized by its values on any sphere $S_{\norm{\cdot}}$, where $\norm\cdot$ is an arbitrary norm on $\R^{d+1}$. By choosing $\norm{\cdot} = \norm{\cdot}_F$ and using the radial symmetry of any $F$-norm, we obtain the following corollary.
\begin{cor}
\label{propsph}
Each $F$-norm $\norm\cdot_F$ on $\R^{d+1}$ is characterized by the part of its unit sphere contained in the positive orthant of $\R^{d+1}$, that is:
\[
S_{\norm\cdot_F}^+:= S_{\norm{\cdot}_F} \cap [0,\infty)^{d+1} = \set{\bfx\ge\bfzero\in\R^{d+1}:\,\norm{\bfx}_F=1}.
\]
\end{cor}
This corollary provides a compact set characterizing any multivariate distribution with nonnegative, nonzero and integrable components. For such distributions, it is therefore an alternative to the lift zonoid studied by~\cite{kosmos1998} and~\cite{mosler2002}. The next two examples show how this set can be computed in practice.
\begin{exam}[Unit sphere for the uniform $F$-norm]
\upshape Let $F$ be the uniform distribution on $(0,1)$. We know from Example~\ref{exunif} that this distribution is characterized by the bivariate $F$-norm given by
\[
\forall x_0,x_1\geq 0, \ \norm{(x_0,x_1)}_F=\begin{cases}
x_0,&\text{if }x_1\le x_0,\\
\dfrac{x_0^2+x_1^2}{2x_1}, &\text{if }x_1>x_0.
\end{cases}
\]
As a consequence, the set $S_{\norm\cdot_F}^+$ corresponding to this norm is the set
\[
S_{\norm\cdot_F}^+ = \set{(1,x_1):\,x_1\in[0,1]} \cup \set{\left( x_0,1+\sqrt{1-x_0^2} \right):\,x_0\in[0,1)}.
\]
This set is represented in Figure~\ref{unitsphunif}.
\end{exam}

\begin{figure}
\begin{center}
\includegraphics[width=12cm,height=12cm]{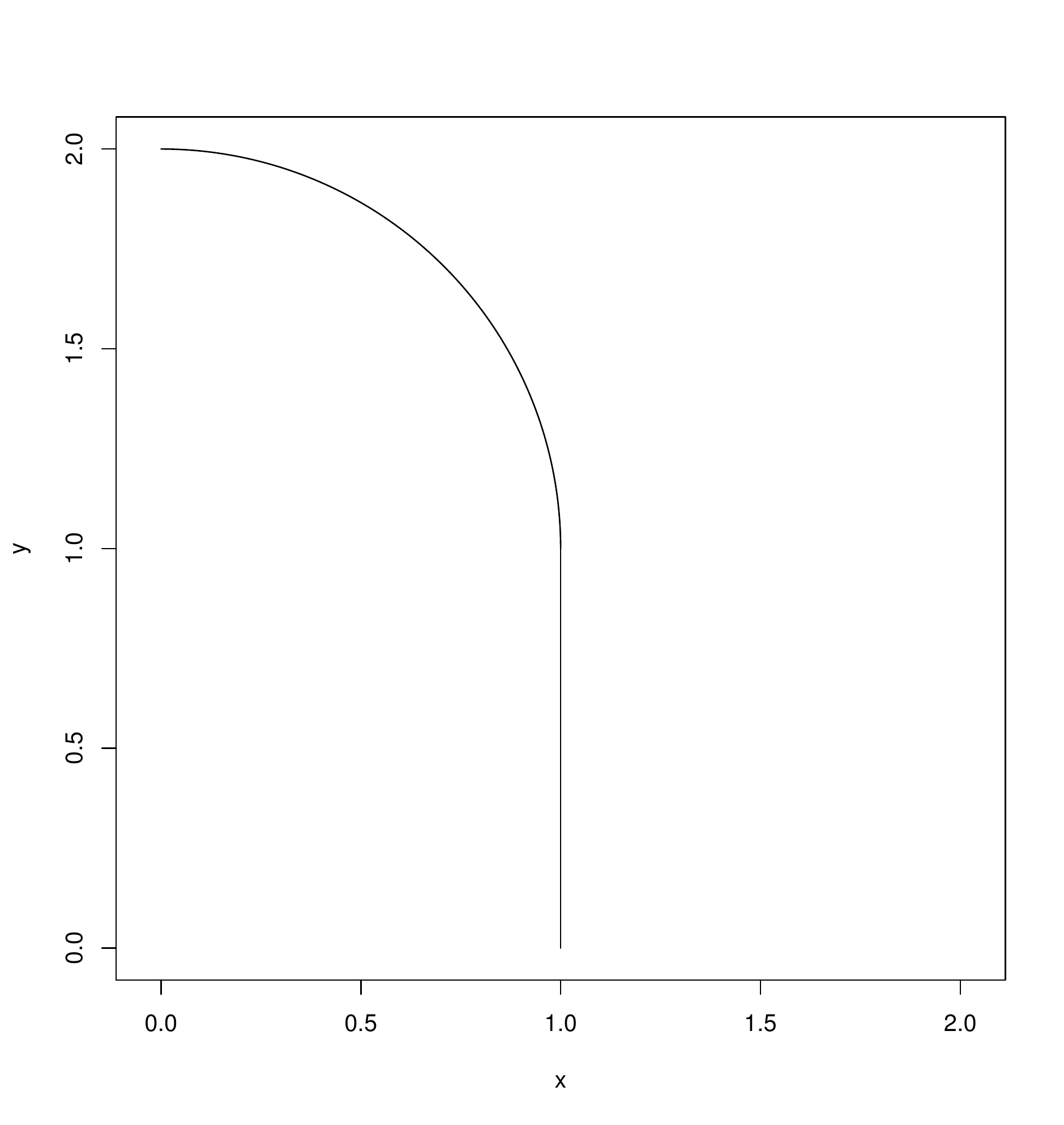}
\end{center}
\caption{The set $S_{\norm\cdot_F}^+$ for the standard uniform distribution.}
\label{unitsphunif}
\end{figure}

\begin{exam}[Unit sphere for the H\"usler-Reiss norm]
\label{exHusRei}
\upshape Let $\norm{\cdot}_F$ be the bivariate H\"usler-Reiss norm with parameter $\sigma^2$, that is
$$
\forall x,y>0, \ \norm{(x,y)}_F = x \Phi\left(\frac{\sigma}{2} + \frac{\log(x/y)}{\sigma} \right) + y \Phi\left(\frac{\sigma}{2} + \frac{\log(y/x)}{\sigma} \right).
$$
Clearly $(1,0)$ and $(0,1)$ belong to $S_{\norm\cdot_F}^+$. If $x,y>0$ are such that $(x,y)\in S_{\norm\cdot_F}^+$ then
$$
\Phi\left(\frac{\sigma}{2} + \frac{\log(x/y)}{\sigma} \right) + \frac{y}{x} \Phi\left(\frac{\sigma}{2} + \frac{\log(y/x)}{\sigma} \right) = \frac{1}{x},
$$
which implies, if $\lambda:=y/x \in (0,\infty)$, that
\begin{align*}
x &= \frac{1}{\Phi\left(\dfrac{\sigma}{2} - \dfrac{\log(\lambda)}{\sigma} \right) + \lambda \, \Phi\left(\dfrac{\sigma}{2} + \dfrac{\log(\lambda)}{\sigma} \right)}, \\[5pt]
\mbox{and } \ y &= \frac{\lambda}{\Phi\left(\dfrac{\sigma}{2} - \dfrac{\log(\lambda)}{\sigma} \right) + \lambda \, \Phi\left(\dfrac{\sigma}{2} + \dfrac{\log(\lambda)}{\sigma} \right)}.
\end{align*}
It is readily checked that conversely, any point of the form
$$
\frac{1}{\Phi\left(\dfrac{\sigma}{2} - \dfrac{\log(\lambda)}{\sigma} \right) + \lambda \, \Phi\left(\dfrac{\sigma}{2} + \dfrac{\log(\lambda)}{\sigma} \right)} (1,\lambda), \ \mbox{ for } \ \lambda\in (0,\infty)
$$
belongs to $S_{\norm\cdot_F}^+$, so that we have a parametrization of $S_{\norm\cdot_F}^+$ making it possible to represent this set. This is done in Figure~\ref{unitsphHR} for various values of $\sigma$. One can observe in this Figure that, as should be apparent from the parametrization, the limit $\sigma\downarrow 0$ produces the part of the sphere of the sup-norm on $\R^2$ contained in the upper right quadrant, while the limit $\sigma\to \infty$ yields the segment $\{(x,1-x), \ 0\leq x\leq 1\}$, corresponding to the sphere of the $L^1-$norm.
\end{exam}

\begin{figure}
\begin{center}
\includegraphics[width=12cm,height=12cm]{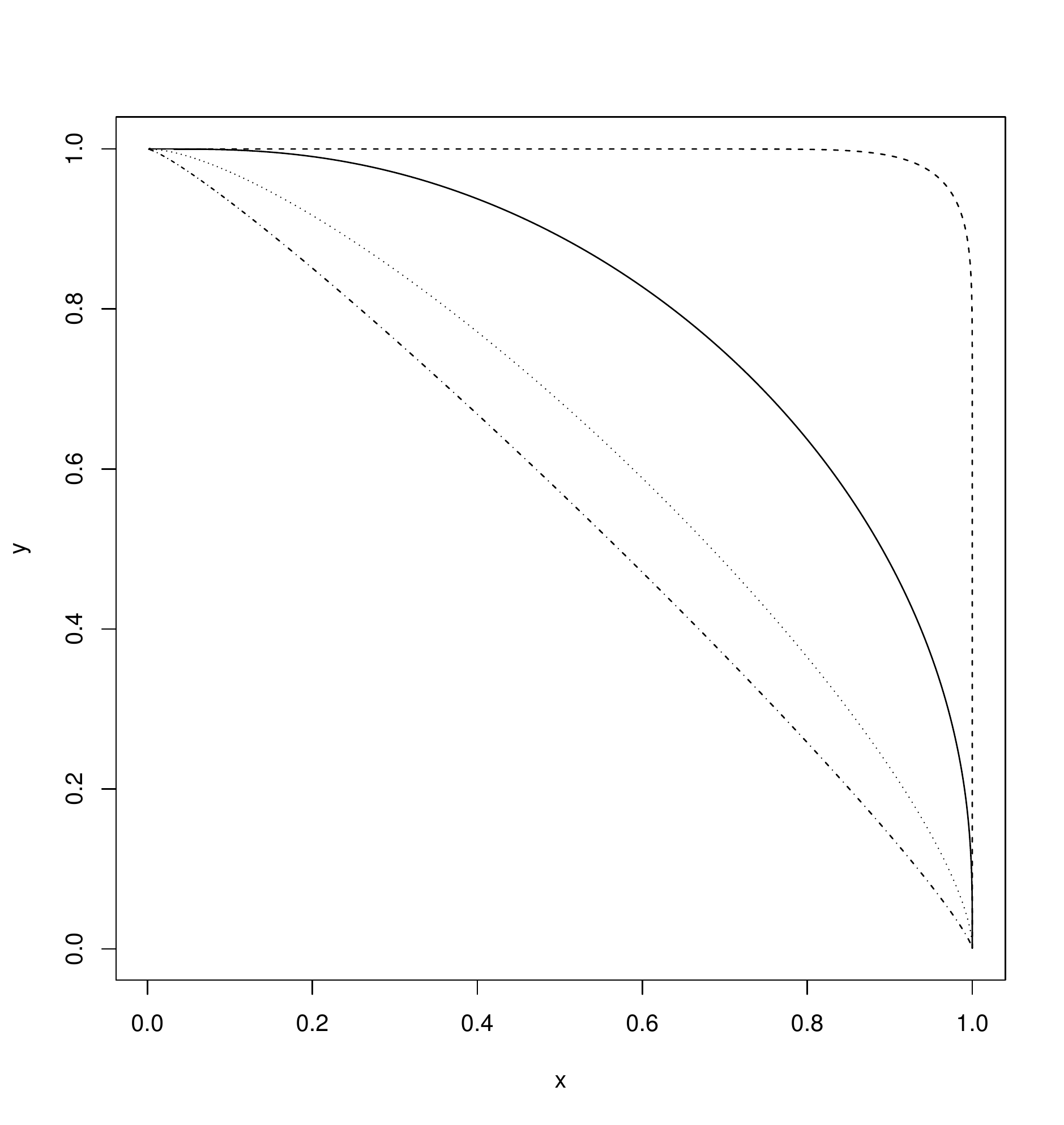}
\end{center}
\caption{The set $S_{\norm\cdot_F}^+$ for the bivariate H\"{u}sler-Reiss $F$-norm. Dashed curve: $\sigma=0.1$, solid curve: $\sigma=1$, dotted curve: $\sigma=2$, dashed-dotted curve: $\sigma=3$.}
\label{unitsphHR}
\end{figure}

Example~\ref{exHusRei} suggests that the convergence of $F$-norms, and thus convergence of the pertaining distributions in the Wasserstein metric, is at least informally linked to the convergence of their unit spheres. To make this intuition rigorous, we recall the definition of a {\it Hausdorff metric}. If $\norm\cdot$ is an arbitrary norm on $\R^{d+1}$ and $A$, $B$ are two subsets of $\R^{d+1}$, we let their $\norm\cdot$-Hausdorff distance to be 
$$
d_{H,\norm\cdot}(A,B)=\max\left\{ \sup_{\bfy\in B} \inf_{\bfx\in A} \| \bfx-\bfy \|, \ \sup_{\bfx\in A} \inf_{\bfy\in B} \| \bfx-\bfy \| \right\}. 
$$
Intuitively, two sets $A$ and $B$ are therefore close in the $\norm\cdot$-Hausdorff metric if and only if each point in $A$ (resp. $B$) is close, in terms of $\norm\cdot$, to at least one point in $B$ (resp. $A$). Such a distance may be infinite under no further assumptions on $A$ and $B$, but is always finite if $A$ and $B$ are bounded. With this definition in mind, we have the following result. 
\begin{theorem}
\label{theoHaus}
Pointwise convergence of a sequence of $F$-norms $\norm\cdot_{F_n}$ to an $F$-norm $\norm\cdot_{F}$ on $\R^{d+1}$ is equivalent to convergence of the sequence of sets $S_{\norm\cdot_{F_{\scalebox{.7}{$\scriptscriptstyle n$}}}}^+$ to $S_{\norm\cdot_F}^+$ in any Hausdorff metric $d_{H,\norm\cdot}$ on $\R^{d+1}$.
\end{theorem}
Our final result, relating convergence of distributions in the Wasserstein metric to convergence of unit spheres of $F$-norms in the Hausdorff metric, is now an immediate corollary of Theorems~\ref{theoWass},~\ref{theoHaus} and the radial symmetry of $F$-norms.
\begin{cor}
\label{corHaus}
If $F_n$, $F$ are multivariate dfs on $\R^d$ with nonnegative, nonzero and integrable components, then the convergence of $F_n$ to $F$ in the Wasserstein metric is equivalent to the convergence of the unit sphere of $\norm{\cdot}_{F_n}$ to the unit sphere of $\norm{\cdot}_F$ in any Hausdorff metric $d_{H,\norm\cdot}$ on $\R^{d+1}$.
\end{cor}

\begin{proof} We start by noting that since all norms are equivalent on $\R^{d+1}$, it is sufficient to prove the theorem for the Hausdorff metric $d_{H,\norm\cdot}$ induced by the norm $\norm\cdot_F$.

Suppose that $\norm\cdot_{F_n} \to \norm\cdot_{F}$ pointwise. Then, by Theorem~\ref{theoWass}, we have $F_n\to F$ in the Wasserstein metric. Let $\bfX^{(n)}$, $\bfX$ have dfs $F_n$ and $F$. This yields 
$$
\forall j\in \{1,\ldots,d\}, \ E\left( X_j^{(n)} \right) \to E\left( X_j \right) \ \mbox{ as } \ n\to\infty, 
$$
and thus, since $\bfX^{(n)}$, $\bfX$ satisfy $(\mathcal{H})$, there is $c>0$ such that 
$E\left( X_j \right)\geq c$ and $E( X_j^{(n)} )\geq c$ for any $n$. Define then a weighted sup-norm $\norm\cdot_{\infty,c}$ on $\R^{d+1}$ by
$$
\norm{(x_0,x_1,\ldots,x_d)}_{\infty,c} := \max(|x_0|, c|x_1|,\ldots,c|x_d|).
$$
By Proposition~\ref{propbounds}, we obtain $\norm\cdot_{F} \geq \norm\cdot_{\infty,c}$ and $\norm\cdot_{F_n} \geq \norm\cdot_{\infty,c}$ for any $n$. Consequently, if $B$ denotes the closed unit ball for the norm $\norm\cdot_{\infty,c}$ and $\mathcal{B} := B\cap [0,\infty)^{d+1}$, then $\mathcal{B}$ contains $S_{\norm\cdot_F}^+$ and the $S_{\norm\cdot_{F_{\scalebox{.7}{$\scriptscriptstyle n$}}}}^+$ for any $n$. In addition, by inequality~\eqref{FnormLip} and since $\mathcal{B}$ is compact,
$$
u_n:= \sup_{\bfx\in \mathcal{B}} | \norm{\bfx}_{F_n} - \norm{\bfx}_F | \leq \sup_{\bfx \in \mathcal{B}} \norm{\bfx}_{\infty} \cdot d_W(F_n,F) \to 0.
$$
Assume from now on that $n$ is so large that $u_n<1$. Pick $\bfx$ in $S_{\norm\cdot_F}^+$. Then since $\mathcal{B}$ contains $S_{\norm\cdot_F}^+$, we have 
$$
| \norm{\bfx}_{F_n} - 1 | = | \norm{\bfx}_{F_n} - \norm{\bfx}_F | \leq \sup_{\bfx\in \mathcal{B}} | \norm{\bfx}_{F_n} - \norm{\bfx}_F | =u_n.
$$
This also entails $\norm{\bfx}_{F_n} \geq 1-u_n$. Note then that $\bfx/\norm{\bfx}_{F_n} \in S_{\norm\cdot_{F_{\scalebox{.7}{$\scriptscriptstyle n$}}}}^+$ and thus 
\begin{equation}
\label{Haus1}
\norm{\bfx - \frac{\bfx}{\norm{\bfx}_{F_n}}}_F = \frac{| \norm{\bfx}_{F_n} - 1 |}{\norm{\bfx}_{F_n}} \leq \frac{u_n}{1-u_n} =:\varepsilon_n.
\end{equation}
If $\bfx$ in $S_{\norm\cdot_{F_{\scalebox{.7}{$\scriptscriptstyle n$}}}}^+$ we have, since $\mathcal{B}$ contains $S_{\norm\cdot_{F_{\scalebox{.7}{$\scriptscriptstyle n$}}}}^+$, 
$$
| \norm{\bfx}_F - 1 | = | \norm{\bfx}_{F_n} - \norm{\bfx}_F | \leq \sup_{\bfx\in \mathcal{B}} | \norm{\bfx}_{F_n} - \norm{\bfx}_F | =u_n.
$$
Write then $\bfx/\norm{\bfx}_F \in S_{\norm\cdot_F}^+$, which yields 
\begin{equation}
\label{Haus2}
\norm{\bfx - \frac{\bfx}{\norm{\bfx}_F}}_F = | \norm{\bfx}_F - 1 | \leq u_n\leq \varepsilon_n.
\end{equation}
From~\eqref{Haus1} and~\eqref{Haus2} it follows that 
$$
d_{H,\norm\cdot_F}\left(S_{\norm\cdot_{F_{\scalebox{.7}{$\scriptscriptstyle n$}}}}^+, S_{\norm\cdot_F}^+ \right)\leq \varepsilon_n\to 0,
$$
showing the convergence of $S_{\norm\cdot_{F_{\scalebox{.7}{$\scriptscriptstyle n$}}}}^+$ to $S_{\norm\cdot_F}^+$ in the Hausdorff metric $d_{H,\norm\cdot_F}$.

Conversely, suppose that $S_{\norm\cdot_{F_{\scalebox{.7}{$\scriptscriptstyle n$}}}}^+ \to S_{\norm\cdot_F}^+$ in the Hausdorff metric $d_{H,\norm\cdot_F}$. By radial symmetry and homogeneity of $F$-norms it is enough to prove the desired pointwise convergence of $\norm{\cdot}_{F_n}$ to $\norm{\cdot}_F$ on $S_{\norm\cdot_F}^+$. Pick then $\bfx\in S_{\norm\cdot_F}^+$. Note that $\bfx/\norm{\bfx}_{F_n}\in S_{\norm\cdot_{F_{\scalebox{.7}{$\scriptscriptstyle n$}}}}^+$ and thus, by assumption, there is a sequence $(\bfz^{(n)})\subset S_{\norm\cdot_F}^+$ with
$$
\norm{\bfz^{(n)}-\frac{\bfx}{\norm{\bfx}_{F_n}}}_F\to 0.
$$ 
By the reverse triangle inequality, this entails 
$$
\left| 1-\frac{1}{\norm{\bfx}_{F_n}} \right| = \left| \norm{\bfz^{(n)}}_F-\frac{\norm{\bfx}_F}{\norm{\bfx}_{F_n}} \right| \leq \norm{\bfz^{(n)}-\frac{\bfx}{\norm{\bfx}_{F_n}}}_F \to 0.
$$
This shows that $1/\norm{\bfx}_{F_n} \to 1$ and thus $\norm{\bfx}_{F_n}\to 1=\norm{\bfx}_F$ as required.
\end{proof}
%

\section*{Acknowledgments}

This research was in large part carried out when M. Falk was visiting G. Stupfler at the University of Nottingham in July 2018. The first author is grateful to his host for his
hospitality and the extremely constructive atmosphere. Support from the London Mathematical Society Research in Pairs Scheme (reference 41710) is gratefully acknowledged.


\end{document}